\documentclass[reqno]{amsart}            
\usepackage{graphicx,cite,color}
\usepackage[bookmarks,bookmarksnumbered,bookmarksopen,colorlinks,backref,linkcolor=blue,citecolor=red]{hyperref}%
\usepackage{xcolor}
\usepackage{dsfont}
\usepackage{mathrsfs}
\usepackage{amsmath, amssymb ,amsthm, amsfonts, amsgen}

\newtheorem{theorem}{Theorem}[section]

\newtheorem{lemma}{Lemma}[section]
\newtheorem{corollary}{Corollary}[section]
\newtheorem{remark}{Remark}[section]
\numberwithin{equation}{section}
\newtheorem{proposition}{Proposition}[section]

\DeclareMathOperator*{\esssup}{ess\,sup}

\begin{document}

\title[Nonlinear convection-dominated flow in cylindric domains]
 {Reduced-dimensional modelling \\
 for nonlinear convection-dominated flow\\  in cylindric domains}

\author[T. Mel'nyk \& C. Rohde]{ Taras Mel'nyk$^{\flat}$ \ \& \ Christian Rohde$^\natural$}
\address{\hskip-12pt
$^\flat$ Institute of Applied Analysis and Numerical Simulation,
the University of Stuttgart\\
Pfaffenwaldring 57,\ 70569 Stuttgart,  \ Germany
 }
\email{Taras.Melnyk@mathematik.uni-stuttgart.de}

\address{\hskip-12pt  $^\natural$ Institute of Applied Analysis and Numerical Simulation,
the University of Stuttgart\\
Pfaffenwaldring 57,\ 70569 Stuttgart,  \ Germany
}
\email{christian.rohde@mathematik.uni-stuttgart.de }

\begin{abstract}
The aim of the paper is to construct and justify  asymptotic approximations for  solutions to  quasilinear    convection-diffusion problems with a predominance of nonlinear convective flow in a thin cylinder, where an inhomogeneous nonlinear Robin-type boundary condition involving convective and diffusive fluxes is imposed on the lateral surface. The limit problem for vanishing diffusion and the cylinder shrinking to an interval is a nonlinear first-order  conservation law. For a time span that allows for a classical solution of this limit problem corresponding uniform pointwise  and energy estimates are proven. They   provide precise model error estimates with respect to the  small parameter  that controls the double viscosity-geometric limit. In addition, other problems with more higher P\'eclet numbers are also considered.
\end{abstract}



\maketitle
\tableofcontents

\section{Introduction}\label{Sect1}

For modelling various complex processes associated with flow and transport in porous media (applications include various processes in thin
networks of cracks in various media,  neural networks, blood vessels in body tissues, root systems in the ground or piping systems),
 three-dimensional  numerical  calculations are often too computationally intensive. Therefore, various methods of model reduction are often used to obtain simplified models of lower dimensions; see e.g., \cite{Ber-Grap-2022,Koch-2022,Wohl-2020,graph-model-2009}  and references therein.

In many cases, these reductions result from significant simplifications and assumptions that do not reflect actual transport processes and are not rigorously justified. The purpose of this article is to give a reasoned answer on how to perform dimensionality reduction in convection-dominated  quasilinear transport problems in thin cylinders and graph-like networks.

\smallskip

We illustrate the relevant issues on a classical mathematical model of salt transport along a cylindric tube of radius $a$ through which water  flows according to Poiseuille's law (see \cite{Taylor_1953}). Namely, the velocity field of the fluid along the $x_1$-axis was  described by a vector-function
$$
\vec{V} = \big(v_0(1 - \tfrac{r^2}{a^2}), 0, 0 \big),
$$
where $r$ is the distance from the central  line of the tube to the cylindrical surface. The wall of the tube was considered impermeable, i.e., $\frac{\partial u}{\partial r} = 0$ at $r = a,$ where $u$ is the concentration of salt.
The concentration distribution has been the subject of analysis under certain assumptions. In particular, it has been demonstrated that the effective concentration obtained through section averaging is determined by the one-dimensional diffusion equation for the longitudinal variable, which is now referred to as \textit{Taylor dispersion}.

Shear flows $\big(v(r), 0, 0 \big)$ are obviously incompressible $(\nabla \boldsymbol{\cdot} \vec{V} = 0)$ and they arise naturally in various physical applications \cite{Maj_Kra_1999}. However, the appearance of roughness in the pipe or the penetration of transported particles through the wall can change the nature of the flow (it can become either non-incompressible or non-laminar). In such cases, the flow should be considered as a non-shear flow and generally as a function of the concentration of particles transported,  which results in the mathematical problem being defined by a typical quasi-linear convection-dominated transport problem such as \eqref{intr.1} below.

\smallskip

Let us be more precise about the mathematical setting  for this article. After non-dimensionalizing a convection-dominated transport problem in a thin cylinder like the described salt transport problem, one  obtains as mathematical model the re-scaled parabolic differential equation
\begin{equation}\label{intr.1}
  \partial_t u_\varepsilon -  \varepsilon\, \Delta_x u_\varepsilon + \nabla_x  \boldsymbol{\cdot} \big(u_\varepsilon\,  \overrightarrow{V_\varepsilon}(u_\varepsilon)\big)
   =  0,
\end{equation}
with a small parameter $\varepsilon$ scaling the second order  diffusion operator. This means that convective processes predominate over diffusive ones
and the P\'eclet number as the ratio of  convective to  diffusive transport rates  is of order $\mathcal{O}(\varepsilon^{-1}).$
Here the diffusion operator is the Laplacian $\Delta_x = \sum_{i=1}^{3}\partial^2_{x_i x_i},$  \, $\partial_{x_i}= \partial/\partial x_i,$ \, $x=(x_1,x_2,x_3),$ \, $\nabla_x = \text{grad},$   \,  $u_\varepsilon$ is the  unknown concentration   of some transported species in a thin cylinder~$\Omega_\varepsilon$ with a small cross-sectional radius  of order $\mathcal{O}(\varepsilon).$ The given convective vector field $\overrightarrow{V_\varepsilon}$  might be  not constant in the cross-section,  but can be assumed to  have  small transversal velocity components of order $\mathcal{O}(\varepsilon),$ and the longitudinal velocity component depends explicitly  on the unknown species concentration $u_\varepsilon.$

To  include boundary interactions with the surrounding bulk domain, we consider  the following nonlinear boundary condition
\begin{equation}\label{intr.2}
- \varepsilon   \partial_{\boldsymbol{\nu}_\varepsilon} u_\varepsilon  +  u_\varepsilon \, \overrightarrow{V_\varepsilon}(u_\varepsilon) \boldsymbol{\boldsymbol{\cdot}}\boldsymbol{\nu}_\varepsilon   =  \varepsilon\,  \varphi_\varepsilon\big(u_\varepsilon,x,t\big)
\end{equation}
on the lateral surface of the thin cylinder, in which both diffusion and convection flows are involved.

As $\varepsilon$ goes to zero, the thin cylinder $\Omega_\varepsilon$ shrinks to an interval, the diffusion operator is eliminated and, as we will show as our main result, the limit concentration $w$ must be a solution of the first-order hyperbolic quasilinear differential equation
\begin{equation}\label{in_3}
\partial_t w + \partial_{x_1}\!\big( v_1(w, x_1, t) \,  w\big) = -\widehat{\varphi}(w, x_1, t),
\end{equation}
where $-\widehat{\varphi}$ is the limit homogenized transformation of $\varphi_\varepsilon.$

But, it is common knowledge (see e.g. \cite{Lax,Lax_1972,Jorn_1974,Oleinik_1957}) that for solutions of such quasilinear equations, discontinuities usually appear after a finite time even for smooth coefficients and given smooth data of the problem.
For our original problem in a tube, this may indicate that turbulence occurs after this time.
Moreover, the solution $w$ cannot simultaneously satisfy the initial and boundary conditions on the ends  of the interval.
This poses a considerable challenge   for the study of such nonlinear problems.

Here, we restrict ourselves to a setting with classical solutions. We demonstrate how  to construct an approximation to the solution of the original problem for a given equation \eqref{intr.1} and the boundary condition~\eqref{intr.2}, using the solution of equation \eqref {in_3}, for a period of time when it still remains smooth.

In addition,  in the last section we explore the extension of these results to problems with different intensity coefficients and diffusion operators, including nonlinear ones, and with more higher P\'eclet numbers. Namely, we consider the  diffusion operator $\varepsilon^\beta \,\partial^2_{x_1 x_1} + \varepsilon \, \Delta_{x_2 x_3}$ and the longitudinal component of $\overrightarrow{V_\varepsilon}$ is of order
$\mathcal{O}(\varepsilon^{\frac{\beta -1}{2}}),$ where  $\beta >1,$ i.e.,  the  P\'eclet number is of  order $\mathcal{O}(\varepsilon^{-\frac{\beta +1}{2}})$ in  the longitudinal direction of the cylinder. As we will see, the limit concentration will now be the solution to the Cauchy problem
\begin{equation}\label{intr.4}
 \partial_t{w}   =  - \widehat{\varphi}({w}, x_1, t),     \qquad     w|_{t=0} =  0,
\end{equation}
where the variable $x_1$ is regarded as a parameter. This case agrees well with physical reality, as the low longitudinal flow magnitude does not cause turbulence in the cylinder throughout the entire time period.

The work can be understood as a continuation of our
papers on  linear convection-dominated non-stationary transport problems in thin graph-like networks (see \cite{Mel-Roh_AnalAppl-2024,Mel-Roh_AsAn-2024}), and on semi-linear ones when the vector field does not depend on the solution (see \cite{Mel-Roh_JMAA-2024}).

Various  transport problems in thin two-dimensional strips, when  the convection is dominating over diffusion, in the regime of Taylor dispersion have been considered in \cite{Pop-2016,Mikelich-2007} (see also references there).
The articles \cite{Macrone-2021,Pazhanin-2011,Pan-Pia-2010} studied stationary (elliptic) transport problems with shear convective flows in thin regions.


The paper has the following structure. We introduce  the precise formulation of our quasilinear problem in a thin cylinder and formulate the main assumptions in Section~\ref{Sec:Statement}.  Section~\ref{Sec:Existence} provides a justification for the solution's existence and explains some of its properties. Formal asymptotic analysis is performed in Section~\ref{Sec:4 Formal analysis}. Here we derive the limit mixed problem for  the first-order hyperbolic quasilinear differential equation \eqref{in_3} and justify the local existence of a smooth solution.
To satisfy the Dirichlet boundary condition at the right base of the thin cylinder, the boundary-layer ansatz  is introduced and the exponential decay of its coefficients is proved. The main our results are presented in Section~\ref{Sec:justification}, where the corresponding uniform pointwise and energy estimates are proven, which provide approximations of the solution with a given accuracy
regarding the small parameter $\varepsilon.$ In the concluding Section~\ref{Sect-Conclusion} we discuss the results and their potential application to similar problems in thin graph-like networks, and extension to other problems.

\section{Problem statement}\label{Sec:Statement}
Let $\varpi \subset \Bbb R^2 $ be a bounded simply connected domain with the smooth boundary that contains the origin.  A thin cylinder
 $\Omega_\varepsilon$  is defined as follows
$$
\Omega_\varepsilon =
  \left\{  x=(x_1, x_2, x_3)\in\Bbb{R}^3 \colon \ 0  <x_1 <\ell, \quad  \varepsilon^{-1} \overline{x}_1 \in \varpi
  \right\},
$$
where  $\varepsilon> 0$  is  a small  parameter and $\overline{x}_1 =  (x_2, x_3).$ Thus, the set $\varpi$ describes the $\varepsilon$-scaled cross section of   $\Omega_\varepsilon$  in any $x_1\in (0,\ell)$.
We  denote by
$
\Gamma_\varepsilon := \partial\Omega_\varepsilon \cap \{ x\colon \ 0<x_1<\ell \}
$
the cylinder's  lateral surface.

The components of the  given vector-valued function
\begin{equation}
   \label{str_1}
\overrightarrow{V_\varepsilon}=
\left(v_1(s, x_1,t), \ \varepsilon  v_2(x_1, \tfrac{\overline{x}_1}{\varepsilon},t), \
 \varepsilon  v_3(x_1,\tfrac{\overline{x}_1}{\varepsilon},t)
\right), \quad (s, x_1,t) \in  \Bbb R \times \overline{\Omega}_\varepsilon\times [0,T],
\end{equation}
are supposed to satisfy the following conditions:
\begin{description}
\item[{\bf A1}]
   $v_2$ and $v_3$  are smooth in $\overline{\Omega}_1\times [0,T]$ and have  compact supports with respect to the  longitudinal  variable $x_1,$  in particular, we will assume that they vanish in  $[0, \delta_1]$ and  $[\ell - \delta_1, \ell],$ where  $\delta_1$ is  sufficiently small positive number;
 \end{description}

\begin{description}
  \item[{\bf A2}]
 $v_1 \in C^\infty\big(\Bbb R\times [0, \ell]\times [0,T]\big),$   is independent of $x_1$   for $x_1\in [0, \delta_1]$ and $t\in [0, \delta_1],$
  $v_1 = \mathrm{v}_1(t) \ge \varsigma_0 = const>0$ for $(x_1,t)\in [\ell - \delta_1, \ell] \times [0,T],$  and
 \begin{equation}\label{v_1}
0\le v_1(s, x_1,t) \le C_0  \quad \text{and}\quad  |s \, \partial_s v_1(s, x_1,t)|+ |s \, \partial^2_{s s} v_1(s, x_1,t)| \le C_1
 \end{equation}
for all \ $ (s,x_1,t)\in \Bbb R \times [0,\ell]\times[0,T].$
\end{description}
So, the main direction of the vector field $\overrightarrow{V_\varepsilon}$ is along the axis of the cylinder~$\Omega_\varepsilon$ from left to right.


In $\Omega_\varepsilon^T := \Omega_\varepsilon \times (0, T),$ where $T > 0,$  we consider the following parabolic convection-diffusion  problem:
\begin{equation}\label{probl}
\left\{\begin{array}{rcll}
 \partial_t u_\varepsilon  -  \nabla_{{x}} \boldsymbol{\cdot} \Big( \varepsilon\, \nabla_{x} u_\varepsilon -
u_\varepsilon \overrightarrow{V_\varepsilon}\Big)
  & = & 0 &
    \text{in} \ \Omega^T_\varepsilon,
\\[2mm]
-   \partial_{\vec{\nu}_\varepsilon} u_\varepsilon +    u_\varepsilon \, \overline{V}_\varepsilon \cdot\vec{\nu}_\varepsilon &  = &  \varphi_\varepsilon\big(u_\varepsilon,x,t\big) &
   \text{on} \ \Gamma_\varepsilon^T,
\\[2mm]
u_\varepsilon\big|_{t=0}&=& 0, & \text{in} \ \Omega_{\varepsilon},
\\[2mm]
 u_\varepsilon\big|_{x_1= 0} = 0,
 &  & u_\varepsilon\big|_{x_1= \ell} = q_\ell(t),  & t \in  [0, T].
 \end{array}\right.
\end{equation}
In \eqref{probl},  we used the transversal velocity function
$$
\overline{V}_\varepsilon:= \big(v_2(x_1, \tfrac{\overline{x}_1}{\varepsilon},t), \, v_3(x_1,\tfrac{\overline{x}_1}{\varepsilon},t)\big),
$$
 the $\varepsilon$-scaled interaction function
$ \varphi_\varepsilon(s,x,t) := \varphi\big(s, x_1, {\overline{x}_1}/{\varepsilon}, t\big)$,
  and  denoted by $\vec{\nu}_\varepsilon \big({\overline{x}_1}/{\varepsilon}\big)$
the outward unit normal to $\Gamma_\varepsilon$. The set $\Gamma_\varepsilon^T$ is given by $\Gamma_\varepsilon^T:= \Gamma_\varepsilon \times (0, T)$.
For the interaction function $\varphi$ and the Dirichlet datum $q_\ell$ at $x_1=\ell $ in problem \eqref{probl}, we suppose
the following conditions.
\def\Cone{{\bf{A3}}}
\def\Ctwo{{\bf{A4}}}
\def\Cthree{{\bf{A5}}}
\begin{description}
  \item[{\Cone}]
The function $\varphi(s,x_1, \overline{x}_1,t)$ with the domain of definition
$$
\mathcal{X}:= \big\{s \in \Bbb R, \ \ x_1 \in  [0, \ell], \ \ \overline{x}_1 \in \overline{\varpi},  \  \  t\in [0,T]\big\}
$$
vanishes if $x_1 \in [0, \delta_1]$ and is independent of $s$ for $x_1\in [\ell - \delta_1, \ell]$; in addition,
it   belongs to the space $C^\infty(\mathcal{X})$  and
\begin{gather}\label{phi_cond}
|\varphi(s,x_1, \overline{x}_1,t)| +  |\nabla_{{x}_1}\varphi(s,x_1, \overline{x}_1,t)| +
 |\partial_t\varphi(s,x,t)|
\le C_2 |s| + C_3,
\\ \label{phi_cond+1}
|\partial_s\varphi(s,x_1, \overline{x}_1,t)|   \le C_4 \quad \text{in} \ \ \mathcal{X}.
\end{gather}
    \item[{\Ctwo}]
The  function $q_\ell$ belongs to $C^\infty([0, T])$ and is nonnegative.

\item[{\Cthree}] To satisfy the zero- and first-order matching conditions in \eqref{probl}, it is necessary to fulfill the relations
\begin{equation}\label{match_conditions}
  q_\ell(0) =   \frac{d q_\ell}{dt}(0)= 0, \quad
 \varphi\big|_{t=0} =0.
\end{equation}
\end{description}

In the following section we will justify the existence of   the solution to  problem \eqref{probl} for each fixed value of the parameter $\varepsilon$.  From a physical point of view, problem \eqref{probl} can be interpreted as a transport process of a substance with the unknown concentration $u_\varepsilon$ in the thin cylinder  $\Omega_\varepsilon$  through its lateral surface
 $\Gamma_\varepsilon$ and moving towards its right base $(x_1 = \ell).$

\medskip

Our goal is to construct the  asymptotic approximation for $u_\varepsilon$  as $\varepsilon \to 0,$ i.e., when the diffusion operator disappears and  the thin cylinder $\Omega_\varepsilon$ shrinks into the segment $ \mathcal{I} :=  \{x\colon  x_1 \in  (0, \ell), \ \overline{x}_1  = (0, 0)\}$.


\section{Existence of the solution}\label{Sec:Existence}

Using results on parabolic semilinear initial-boundary problems \cite{Lad_Sol_Ura_1968} and approaches developed there, we show the existence and uniqueness of  a solution to problem~\eqref{probl} for any fixed $\varepsilon > 0$ and some of its properties.
For our problem, the conditions of some theorems from \cite{Lad_Sol_Ura_1968} are not satisfied. Therefore, we will go on to prove the necessary analogues and indicate the necessary changes in the proofs. Where they coincide, we will refer to the original proofs.

By using the substitution $u_\varepsilon = \mathrm{w}_\varepsilon + \frac{x_1}{\ell} q_\ell$,  problem  \eqref{probl} is reduced to
\begin{equation}\label{probl+red}
\left\{\begin{array}{rcll}
 \partial_t \mathrm{w}_\varepsilon  -  \partial_{x_1}\!\Big( \varepsilon\, \partial_{x_1} \mathrm{w}_\varepsilon -
\big(\mathrm{w}_\varepsilon + g\big) \, v_1(\mathrm{w}_\varepsilon + g,x_1,t)\Big) & & &
 \\[2mm]
 -  \varepsilon\, \nabla_{\overline{x}_1} \boldsymbol{\cdot} \Big(\nabla_{\overline{x}_1} \mathrm{w}_\varepsilon -
\big(\mathrm{w}_\varepsilon + g\big) \, \overline{V}_\varepsilon\Big)
  & = & - \partial_t g &
    \text{in} \ \Omega_\varepsilon^T,
\\[2mm]
-   \partial_{\vec{\nu}_\varepsilon} \mathrm{w}_\varepsilon +    \big(\mathrm{w}_\varepsilon + g\big) \, \overline{V}_\varepsilon \cdot \vec{\nu}_\varepsilon &  = &  \Phi_\varepsilon(\mathrm{w}_\varepsilon,x,t) &
   \text{on} \ \Gamma_\varepsilon^T,
\\[2mm]
\mathrm{w}_\varepsilon\big|_{t=0}&=& 0, & \text{in} \ \Omega_{\varepsilon},
\\[2mm]
 \mathrm{w}_\varepsilon\big|_{x_1= 0} = \mathrm{w}_\varepsilon\big|_{x_1= \ell} &=&0,  &
 \end{array}\right.
\end{equation}
where $\nabla_{\overline{x}_1} := \big(\partial_{x_2}, \partial_{x_3}\big),$ $g(x_1,t):= \frac{x_1}{\ell} q_\ell(t),$ $\Phi_\varepsilon(\mathrm{w}_\varepsilon,x, t) := \varphi_\varepsilon(\mathrm{w}_\varepsilon + g, x, t).$

Denote by
$\mathcal{H}_\varepsilon^*$ the dual space to the Sobolev space
$$
\mathcal{H}_\varepsilon = \big\{u \in W^{1}_{2}(\Omega_\varepsilon)\colon  u|_{\Upsilon_\varepsilon(0)} = u|_{\Upsilon_\varepsilon(\ell)} =0\big\},
$$
with ${\|u\|}_{\mathcal{H}_\varepsilon}:= {\|\nabla_x u\|}_{L^2(\Omega_\varepsilon)}.$
Here $u|_{\Upsilon_\varepsilon(0)}$ and $u|_{\Upsilon_\varepsilon(\ell)}$ are the traces of $u$ on the
 bases
$$
\Upsilon_\varepsilon(0):=  \left\{x\colon \ x_1=0, \ \   \varepsilon^{-1} \overline{x}_1 \in \varpi \right\} \ \ \text{and} \ \
\Upsilon_\varepsilon(\ell):=  \left\{x \colon \ x_1=\ell, \ \   \varepsilon^{-1} \overline{x}_1 \in \varpi \right\}
$$
of the cylinder $\Omega_\varepsilon.$

Recall that a function
$\mathrm{w}_\varepsilon$ from the Banach space
$$
\mathcal{W}_\varepsilon:=\big\{u \in L^2 \left(0, T; \, \mathcal{H}_\varepsilon \right)\colon\ \ \partial_t u \in L^2 \left(0, T; \, \mathcal{H}_\varepsilon^* \right)\big\}
$$
equipped with the norm
$$
{\|u\|}_{\mathcal{W}_\varepsilon} = {\|u\|}_{L^2 \left(0, T; \, \mathcal{H}_\varepsilon \right)} +
{\|\partial_t u\|}_{L^2\left(0, T; \, \mathcal{H}_\varepsilon^* \right)}
$$
is called a weak solution to  problem~(\ref{probl+red}) if it satisfies  $\mathrm{w}_\varepsilon|_{t=0}=0$ and  the integral identity
\begin{equation}\label{int-identity}
          \langle\partial_t \mathrm{w}_\varepsilon,  \eta\rangle_\varepsilon
 +   \int\limits_{\Omega_\varepsilon}\sum_{i=1}^{3}a_i\big(x,t,\mathrm{w}_\varepsilon, \partial_{x_i}\mathrm{w}_\varepsilon\big) \partial_{x_i}\eta(x) \, dx
  +  \varepsilon  \int\limits_{\Gamma_\varepsilon} \Phi_\varepsilon(\mathrm{w}_\varepsilon,x, t)\, \eta(x) \, d\sigma_x
 = - \int\limits_{\Omega_\varepsilon} \partial_t g \, \eta(x) \, dx
\end{equation}
for any function $\eta \in \mathcal{H}_\varepsilon$ and a.e. $t\in(0, T).$
Here $\langle\cdot,\cdot\rangle_\varepsilon$ is the duality pairing of $\mathcal{H}_\varepsilon^*$ and $\mathcal{H}_\varepsilon.$
It is well-known that the space $\mathcal{W}_\varepsilon$ is continuously embedded in the space $ C\big([0,T]; L^2(\Omega_\varepsilon)\big),$ i.e.  $\mathrm{w}_\varepsilon\in C\big([0,T]; L^2(\Omega_\varepsilon)\big),$
and thus the equality $\mathrm{w}_\varepsilon|_{t=0}=0$ makes sense. In integrals over the surface $\Gamma_\varepsilon,$
we understand the  integrands as traces of the respective functions.
In \eqref{int-identity}
\begin{gather}\label{coeff+1}
  a_1(x,t,u,p_1):= \varepsilon p_1 -   \big(u + g(x_1,t)\big) \, v_1(u + g(x_1,t), x_1,t)
   \\ \label{coeff+2}
 a_i(x,t,u,p_i):= \varepsilon p_i - \varepsilon \big(u + g(x_1,t)\big) \, v_i(x_1, \tfrac{\overline{x}_1}{\varepsilon},t), \quad i\in \{2, 3\}.
\end{gather}
Note that the function $a_1$ is linear in $p_1$ and nonlinear in $u$, and   $a_i$ is linear in $p_i$ and $u$.

\begin{lemma}
 There exists a  unique weak solution to problem~\eqref{probl+red} for any fixed $\varepsilon > 0.$
 In addition, $\mathrm{w}_\varepsilon$ is a classical solution, e.g., it is continuous in $\overline{\Omega_\varepsilon} \times [0,T]$
 and belongs to the space $C^{2,1}\big(\Omega_\varepsilon^T\big) \cap C^{1,0}\big(\Omega_\varepsilon^T \cup \Gamma^T_\varepsilon\big).$
 \end{lemma}
\begin{remark}
Here and further in this section, constants in all inequalities are positive and independent of the solution $\mathrm{w}_\varepsilon.$ If a constant depends on the parameter $\varepsilon,$ this is indicated in brackets next to the constant. Mostly constants with the same indices in different inequalities are different.
\end{remark}
\begin{proof}
 {\bf 1.} First, we prove a priori estimates \eqref{apriori} and \eqref{func+1}  for a weak solution  $\mathrm{w}_\varepsilon.$
 The existence of the last integral in \eqref{int-identity} it follows from assumptions ${\Cone},$ ${\Ctwo}$  and the identity from \cite[\S 2]{M-MMAS-2008}
  \begin{equation}\label{identity}
\varepsilon \int_{\Gamma_{\varepsilon}} \phi(x) \, d\sigma_x
= \frac{\upharpoonleft\!\!\partial\varpi\!\!\upharpoonright_1}{\upharpoonleft\!\!\varpi\!\!\upharpoonright_2} \int_{\Omega_{\varepsilon}}  \phi \,dx +
\varepsilon  \int_{\Omega_{\varepsilon}}
\nabla_{\overline{\xi}_1}Y(\overline{\xi}_1)|_{\overline{\xi}_1=\frac{\overline{x}_1}{\varepsilon}}
\cdot \nabla_{\overline{x}_1}\phi \, dx
 \end{equation}
for all $\phi \in W^{1}_2(\Omega_{\varepsilon}).$
Here $\upharpoonleft\!\!S\!\!\upharpoonright_n$  is the $n$-dimensional Lebesgue measure of a set $S,$
the function $Y$ is a
unique solution of the following problem
\begin{equation*}
\Delta_{\overline{\xi}_1} Y(\overline{\xi}_1)= \frac{\upharpoonleft\!\!\partial\varpi\!\!\upharpoonright_1}{\upharpoonleft\!\!\varpi\!\!\upharpoonright_2} \;
 \; \; \mbox{in} \; \;
\varpi, \quad \partial_{\vec{\nu}(\overline{\xi}_1)}Y=1 \; \; \mbox{on} \; \;
\partial\varpi, \quad \int_{\varpi}Y(\overline{\xi}_1)\, d\overline{\xi}_1=0,
\end{equation*}
 $\vec{\nu}(\overline{\xi}_1)=\bigl(\nu_2(\overline{\xi}_1), \nu_3(\overline{\xi}_1)\bigr)$
is the outward normal to $\partial\varpi.$ From \eqref{identity}  it follows (see \cite[\S 2]{M-MMAS-2008}) the inequality
\begin{equation}\label{ineq1}
      \varepsilon \int_{\Gamma_\varepsilon} u^2 \, d\sigma_x \leq C \Bigg( \varepsilon^2
      \int_{\Omega_\varepsilon}|\nabla_{\overline{x}_1} u|^2 \, dx
 +    \int_{\Omega_\varepsilon} u^2 \, dx \Bigg) \quad \forall \, u\in W^{1}_{2}(\Omega_\varepsilon),
\end{equation}
This inequality and  \eqref{phi_cond} yield
\begin{equation}\label{phi_int}
\varepsilon  \Big|\int_{\Gamma_\varepsilon} \Phi_\varepsilon(\mathrm{w}_\varepsilon,x, t) \, \mathrm{w}_\varepsilon \, d\sigma_x\Big| \le \varepsilon^2 C_1  \int_{\Omega_\varepsilon}|\nabla_{\overline{x}_1} \mathrm{w}_\varepsilon|^2 \, dx
 +    C_2 \int_{\Omega_\varepsilon} |\mathrm{w}_\varepsilon|^2 \, dx.
\end{equation}

Using the boundedness of $v_1$ and  Cauchy's inequality with $\delta>0 \ (ab \leq \delta a^2 + \tfrac{b^2}{4\delta}),$  we get that
\begin{equation}\label{identity+3}
  \Big|\int_{\Omega_\varepsilon} \big(u + g\big) \, v_1(u + g, x_1,t) \, \partial_{x_1}u \, dx\Big| \le
    \varepsilon \delta  \int_{\Omega_\varepsilon} |\partial_{x_1}u|^2 \, dx
         + \frac{C_0}{\varepsilon \delta} \int_{\Omega_\varepsilon} u^2 \, dx + C_3 \varepsilon
\end{equation}
for any $ u\in \mathcal{H}_\varepsilon.$

Taking $\eta = \mathrm{w}_\varepsilon$ in \eqref{int-identity} and then integrating it over the interval $(0, \tau)$ with respect to $t$, where $\tau$ is any number from $(0, T],$ we get
 \begin{equation}\label{identity+1}
     \frac{1}{2}\int\limits_{\Omega_\varepsilon}
     \mathrm{w}^2_\varepsilon(x,\tau) \, dx
 +   \int\limits_{\Omega^{\tau}_\varepsilon}\sum_{i=1}^{3}a_i\big(x,t,\mathrm{w}_\varepsilon, \partial_{x_i}\mathrm{w}_\varepsilon\big) \, \partial_{x_i}\mathrm{w}_\varepsilon\, dx dt
   +  \varepsilon  \int\limits_{\Gamma^{\tau}_\varepsilon} \Phi_\varepsilon(\mathrm{w}_\varepsilon,x, t) \, \mathrm{w}_\varepsilon \, d\sigma_x dt
      = - \int\limits_{\Omega^{\tau}_\varepsilon} \partial_t g \, \mathrm{w}_\varepsilon\, dx dt,
\end{equation}
where $\Omega^{\tau}_\varepsilon = \Omega_\varepsilon \times(0, \tau), \ \Gamma^{\tau}_\varepsilon = \Gamma_\varepsilon \times(0, \tau).$

With the help of \eqref{phi_int} and \eqref{identity+3},  taking $\varepsilon$ such that $\varepsilon C_1 \le \frac{1}{4}$ in \eqref{phi_int} and applying Cauchy inequality with a suitable  $\delta,$  we derive from  \eqref{identity+1}  the inequality
\begin{equation}\label{ineq+1}
  \int_{\Omega_\varepsilon}  \mathrm{w}^2_\varepsilon(x,\tau) \, dx
 +  \varepsilon \int_{\Omega^{\tau}_\varepsilon} |\nabla_x \mathrm{w}_\varepsilon(x,t)|^2 \, dxdt \le
 \frac{C_1}{\varepsilon} \int_{\Omega^\tau_\varepsilon} \mathrm{w}^2_\varepsilon(x,t)\, dx dt + C_2\, \varepsilon \, \tau.
\end{equation}

Gronwall's lemma applied to the inequality
$$
\int_{\Omega_\varepsilon}  \mathrm{w}^2_\varepsilon(x,\tau) \, dx
 \le
 \frac{C_1}{\varepsilon} \int_{0}^{\tau} \int_{\Omega_\varepsilon} \mathrm{w}^2_\varepsilon(x,t) \, dx dt + C_2 \, \varepsilon\, \tau
$$
gives the estimate
$$
\int_{\Omega^\tau_\varepsilon} \mathrm{w}^2_\varepsilon(x,t) \, dx dt \le \frac{C_2}{C_1}\, \varepsilon^2 \, T \, \big( \exp(\tfrac{C_1}{\varepsilon} T) -1\big),
$$
which together with \eqref{ineq+1} leads to the inequality
\begin{equation}\label{apriori}
  \max_{t\in [0, T]}{\|\mathrm{w}_\varepsilon\|}_{L^2(\Omega_\varepsilon)}  +   {\|\nabla_x \mathrm{w}_\varepsilon\|}_{L^2(\Omega^T_\varepsilon)} \le C_0 \sqrt{T \, \exp(\tfrac{C_1}{\varepsilon} T)}.
 \end{equation}

Well-known embedding theorem (see e.g. \cite[Chapt. II, \S 2, \S3]{Lad_Sol_Ura_1968})  together with \eqref{apriori} yield the inequality
\begin{equation*}
  { \|\mathrm{w}_\varepsilon\|}_{L^{q,r}(\Omega^T_\varepsilon)} := \bigg(\int_{0 }^{T}\Big(\int_{\Omega_\varepsilon} |\mathrm{w}_\varepsilon|^q dx\Big)^{\frac{r}{q}} dt\bigg)^{\frac{1}{r}}
      \le C_1\Big(\max_{t\in [0, T]} {\|\mathrm{w}_\varepsilon\|}_{L^2(\Omega_\varepsilon)}  +   \|\nabla_x \mathrm{w}_\varepsilon\|_{L^2(\Omega^T_\varepsilon)}\Big),
\end{equation*}
where $q\in [2, 6],$ $r\in [2,+\infty]$ and
$
  \frac{1}{r} + \frac{3}{2q} = \frac{3}{4}.
$
From the last equality, if $q=r,$ we deduce that $q=\frac{10}{3}.$ Therefore,
\begin{equation}\label{estim+2}
\|\mathrm{w}_\varepsilon\|_{L^{\frac{10}{3}}(\Omega^T_\varepsilon)} \le C_1(\varepsilon).
\end{equation}

Take any $\eta \in \mathcal{H}_\varepsilon,$ with ${\|\eta\|}_{\mathcal{H}_\varepsilon}:= {\|\nabla_x \eta\|}_{L^2(\Omega_\varepsilon)} \le 1.$ Owing to assumptions ${\bf A1},$ ${\bf A2}$ and ${\Ctwo},$ for each $i\in\{1, 2, 3\}$
we obtain
\begin{equation}\label{v1+1}
  |a_i(x,t,u,p_i)| \le \varepsilon |p_i| + C_1 |u| +C_2 \quad \text{for all} \ \ (x,t) \in \overline{\Omega_\varepsilon^T}, \ u\in \Bbb R, \ \ p_i\in \Bbb R.
\end{equation}
Using these inequalities and \eqref{ineq1}, we derive from \eqref{int-identity} the inequality
$$
|\langle\partial_t \mathrm{w}_\varepsilon,  \eta\rangle_\varepsilon| \le C_1 {\|\mathrm{w}_\varepsilon\|}_{\mathcal{H}_\varepsilon}
  + C_2 \varepsilon  \quad \text{for any} \ \ \eta \in \mathcal{H}_\varepsilon, \ \ {\|\eta\|}_{\mathcal{H}_\varepsilon} \le 1.
$$
This means that  $ {\|\partial_t \mathrm{w}_\varepsilon\|}_{\mathcal{H}_\varepsilon^*} \le C_1 {\|\mathrm{w}_\varepsilon\|}_{\mathcal{H}_\varepsilon}
  + C_2\varepsilon$
holds, and consequently
\begin{equation}\label{func+1}
\|\partial_t \mathrm{w}_\varepsilon\|^2_{L^2 \left(0, T; \, \mathcal{H}_\varepsilon^* \right)} = \int_{0}^{T}\|\partial_t \mathrm{w}_\varepsilon\|^2_{\mathcal{H}_\varepsilon^*} dt \le C_1  \|\nabla_x \mathrm{w}_\varepsilon\|^2_{L^2(\Omega^T_\varepsilon)}
  + C_2\varepsilon^2 T \le C_3(\varepsilon).
\end{equation}

\smallskip

{\bf 2.} For the  proof of the existence of a weak solution to problem \eqref{probl+red} we can follow  the proof of Theorem 6.7 \cite[Chapt.~V, \S 6]{Lad_Sol_Ura_1968} closely,  with the only difference being the choice of a fundamental system $\{\Psi_m\}_{m\in \Bbb N}$ in the space $\mathcal{H}_\varepsilon$ to apply Galerkin's method to our problem. In our case, these can be eigenfunctions of the spectral problem
\begin{equation}\label{SP}
  -\Delta_x \Psi = \lambda \Psi  \ \ \text{in} \ \Omega_\varepsilon, \quad
\partial_{\overline{\nu}_\varepsilon} \Psi= 0 \ \  \text{on} \ \Gamma_\varepsilon, \quad   \Psi|_{x_1 =0}= \Psi|_{x_1 =\ell} =0,
\end{equation}
which have the form
$$
\Psi(x) = \sqrt{\frac{2}{\ell}} \, \sin\Big(\frac{\pi n}{\ell} \, x_1\Big)\, \Theta_k\Big(\frac{\overline{x}_1}{\varepsilon}\Big) , \quad n \in \Bbb N, \  \ k \in \Bbb N_0:= \Bbb N \cup \{0\},
$$
where $\{\Theta_k(\overline{\xi}_1), \  \overline{\xi}_1=(\xi_2, \xi_3) \in \varpi \}_{k \in \Bbb N_0}$ are orthonormal in $L^2(\varpi)$ eigenfunctions of the spectral Neumann problem
$$
-\Delta_{\overline{\xi}_1} \Theta  = \tau  \, \Theta   \ \ \text{in} \ \varpi, \quad
\partial_{\bar{\nu}_{\bar{\xi}_1}} \Theta = 0 \ \  \text{on} \ \partial\varpi.
$$
We arrange eigenvalues of problem \eqref{SP} so they form a non-decreasing sequence: $0 < \lambda_1 < \lambda_2 \le \ldots \le \lambda_m \le \ldots \to +\infty,$ and denote by $\Psi_m$ the eigenfunction corresponding to the eigenvalue $\lambda_m.$ Clearly,
$(\Psi_m, \Psi_k)_{L^2(\Omega_\varepsilon)} = \delta_{m, k}$ and
$(\Psi_m, \Psi_k)_{\mathcal{H}_\varepsilon} = \lambda_k\delta_{m, k},$ where $\delta_{m, k}$ is the Kronecker delta.
In addition, since the boundary $\partial \varpi$ is smooth, the eigenfunctions $\{\Psi_m\}_{m\in \Bbb N}$ are smooth in $\overline{\Omega}_\varepsilon$ and form a basis in the space $W^{2r}_2(\Omega_\varepsilon) \cap \mathcal{H}_\varepsilon$ for each $r\in \Bbb N.$

For any $N\in \Bbb N$ we seek an approximation in the usual form
\begin{equation*}
  U_N(x,t) := \sum_{m=1}^{N} d_{m, N}(t) \, \Psi_m(x),
\end{equation*}
where the coefficients $d_{m, N}(t), \ t\in[0,T], \ m\in\{1,\ldots,N\},$ are solutions of the system of ordinary differential equations
\begin{equation}\label{system}
 \int\limits_{\Omega_\varepsilon}\partial_t  U_N \,    \Psi_k \, dx
 +   \int\limits_{\Omega_\varepsilon}\sum_{i=1}^{3}a_i\big(x,t,U_N, \partial_{x_i}U_N\big) \partial_{x_i}\Psi_k \, dx
  +  \varepsilon  \int\limits_{\Gamma_\varepsilon}
  \Phi_\varepsilon\big(U_N, x, t\big) \, \Psi_k \, d\sigma_x
 =- \int\limits_{\Omega_\varepsilon}\partial_t g \, \Psi_k \, dx,
 \end{equation}
$$
d_{k, N}(0)=0,   \quad k\in \{1,\ldots,N\},
$$
which can be rewritten as
\begin{equation}\label{system1}
  d^{\, \prime}_{k, N}(t) + L_k(t, \vec{\mathbf{d}}_N) = f_k(t),
  \quad
d_{k, N}(0)=0,   \quad k\in \{1,\ldots,N\},
\end{equation}
where $\vec{\mathbf{d}}_N= \big(d_{1, N}(t),\ldots,d_{N, N}(t)\big),$ \ $f_k(t) = - \int_{\Omega_\varepsilon} \partial_t g \, \Psi_k \, dx,$
  \begin{align*}
    L_k(t, \vec{\mathbf{d}}_N) & :=
    \varepsilon\, \lambda_k\, d_{k, N}(t)
    -
  \int_{\Omega_\varepsilon}\big(U_N + g\big) \, v_1(U_N + g,x_1,t) \, \partial_{x_1}\Psi_k\,dx
    \\
    & - \varepsilon
  \int_{\Omega_\varepsilon}  \big(U_N + g\big) \, \overline{V}_\varepsilon \cdot \nabla_{\overline{x}_1}\Psi_k\,  dx
         + \int_{\Gamma_\varepsilon} \Phi_\varepsilon\big(U_N , x, t\big) \, \Psi_k \, d\sigma_x.
  \end{align*}

Define the vector-function $\overrightarrow{\mathbf{L}}\colon [0,T]\times\Bbb R^N \mapsto \Bbb R^N$ as follows
$$
\overrightarrow{\mathbf{L}}= \big( L_1(t, \vec{\mathbf{d}}),\ldots, L_N(t, \vec{\mathbf{d}})\big), \quad  \vec{\mathbf{d}} \in \Bbb R^N.
$$
Thanks to the smoothness of the functions $a_1, a_2, a_3$ and $\varphi_\varepsilon,$ the function $\overrightarrow{\mathbf{L}}$ is continuous in the  domain of definition, and similar as in the first item of the proof, using \eqref{phi_int} and \eqref{identity+3},  we deduce that
$$
\overrightarrow{\mathbf{L}}(t, \vec{\mathbf{d}}) \cdot \vec{\mathbf{d}} \ge -  C_1 \, |\vec{\mathbf{d}}|^2 - C_2 \quad \text{for all} \ \
t\in [0, T], \ \ \vec{\mathbf{d}} \in \Bbb R^N.
$$
Then by Theorem 3 \cite{Kato-1976} (see also Lemma 4 \cite{Dubinski-1968}) the system   \eqref{system1} has a solution defined for all $t\in [0,T].$

Multiplying equation \eqref{system1} by  $d_{k, N}$ and summing for $k\in \{1,\ldots,N\},$ we get \eqref{identity+1} for $U_N.$
Therefore, elements of the sequence $\{U_N\}_{N\in \Bbb N}$ satisfy estimates \eqref{apriori}, \eqref{estim+2} and \eqref{func+1},  where the constants are independent of~$N.$ Consequently there exists a subsequence of $\{U_N\}_{N\in \Bbb N},$
again denoted by $\{U_N\}_{N\in \Bbb N},$ and a function  $\mathrm{w}_\varepsilon\in \mathcal{W}_\varepsilon \cap
C\left([0, T]; \, L^2(\Omega_\varepsilon)\right)$  such that
\begin{gather}\label{conv1}
 U_N   \rightharpoonup \mathrm{w}_\varepsilon \quad \text{weakly in} \ \ L^2 \left(0, T; \, \mathcal{H}_\varepsilon\right),
  \\ \label{conv2}
  \partial_t U_N \rightharpoonup \partial_t\mathrm{w}_\varepsilon \quad \text{weakly in} \ \ L^2 \left(0, T; \, \mathcal{H}_\varepsilon^* \right) \quad \text{as} \ \ N\to +\infty.
\end{gather}
According the compactness lemma by Lions-Aubin \cite{Aubin,Lions}
\begin{equation}\label{conv3}
  U_N   \rightarrow  \mathrm{w}_\varepsilon \quad \text{strongly in} \ \ L^2(\Omega_\varepsilon^T) \ \ \text{and in } \
C([0,T]; \mathcal{H}_\varepsilon^*),
\end{equation}
and so we can regard that
\begin{equation}\label{conv4}
  U_N   \rightarrow  \mathrm{w}_\varepsilon \quad \text{almost everywhere in} \ \ \Omega_\varepsilon^T.
\end{equation}

Now we show that $\mathrm{w}_\varepsilon$ is a weak solution to problem~\eqref{probl}. First we note  from \eqref{conv1},  \eqref{conv2} and the fact  that $U_N|_{t=0} =0$ it follows that $\mathrm{w}_\varepsilon|_{t=0} =0.$

Consider the set of functions
$$
\mathfrak{M}:= \left\{\Xi = \sum_{k=1}^{M} d_k(t) \, \Psi_k(x),\ t\in[0,T], \ x\in \overline{\Omega}_\varepsilon\colon \  d_k\in C^1([0,T]), \quad M\in \Bbb N\right\}
$$
which  is obviously dense in  the space $L^2 \left(0, T; \, \mathcal{H}_\varepsilon \right).$ Then we deduce from \eqref{system} the identity
\begin{multline}\label{identity+4}
   \int_{\Omega_\varepsilon^T}\partial_t  U_N \, \Xi \, dx dt + \varepsilon \int_{\Omega_\varepsilon^T}\nabla_x  U_N \cdot   \nabla_x \Xi \, dx dt
    - \varepsilon \int_{\Omega_\varepsilon^T}    (U_N + g) \, \overline{V}_\varepsilon \cdot \nabla_{\overline{x}_1} \Xi \, dx dt
  \\
- \int_{\Omega_\varepsilon^T} (U_N + g) \, v_1(U_N + g) \, \partial_{x_1}\Xi\, dx dt
+   \varepsilon  \int_{\Gamma_\varepsilon^T} \Phi_\varepsilon(U_N, x, t) \, \Xi \, d\sigma_x dt
\\ = -  \int_{\Omega_\varepsilon^T} \partial_t g \, \Xi \, dx dt \quad \text{for all} \ \ \Xi \in \mathfrak{M}.
\end{multline}
Using the relations \eqref{conv1} - \eqref{conv3},  we can pass  to the limit in \eqref{identity+4} as $N \to +\infty.$ It remains to find the limits of the integrals in the second line of~\eqref{identity+4}.

From \eqref{v_1} and \eqref{estim+2} we deduce that
\begin{equation*}
  \int_{\Omega_\varepsilon^T} \Big|(U_N + g) \, v_1(U_N + g,x_1,t)\Big|^{\frac{10}{3}} dx dt
\le C_1(\varepsilon).
\end{equation*}
This inequality and convergence \eqref{conv4} are sufficient conditions in Lemma 2.3 \cite[Chapt. II, \S2]{Lad_Sol_Ura_1968}, which states that the function sequence
$$
\{(U_N + g) \, v_1(U_N + g,x_1,t)\}_{N\in\Bbb N}
$$
converges to  the function $(\mathrm{w}_\varepsilon + g) \, v_1(\mathrm{w}_\varepsilon + g,x_1,t)$
 in the space $L^2(\Omega_\varepsilon^T).$

 To pass to the limit in the other integral, we use identity \eqref{identity}. Then
 \begin{align}\label{limit+2}
    \varepsilon  \int_{\Gamma_\varepsilon^T} \Phi_\varepsilon(U_N, x, t) \, \Xi \, d\sigma_x dt
   \ = &\  \frac{\upharpoonleft\!\!\partial\varpi\!\!\upharpoonright_1}{\upharpoonleft\!\!\varpi\!\!\upharpoonright_2} \int_{\Omega_{\varepsilon}^T}\varphi_\varepsilon(U_N + g, x, t) \, \Xi \, dxdt \notag
   \\
  & \ +  \varepsilon  \int_{\Omega_{\varepsilon}^T}\varphi_\varepsilon(U_N + g, x, t) \,
\nabla_{\overline{\xi}_1}Y(\overline{\xi}_1)|_{\overline{\xi}_1=\frac{\overline{x}_1}{\varepsilon}}
\cdot \nabla_{\overline{x}_1}\Xi \, dxdt \notag
\\
& \ +  \int_{\Omega_{\varepsilon}^T}\Xi \,
\nabla_{\overline{\xi}_1}Y(\overline{\xi}_1)|_{\overline{\xi}_1=\frac{\overline{x}_1}{\varepsilon}}
\cdot \nabla_{\overline{\xi}_1}\varphi(U_N + g, x_1,\overline{\xi}_1, t)|_{\overline{\xi}_1=\frac{\overline{x}_1}{\varepsilon}}\,  dx dt \notag
\\
&  \ + \varepsilon
\int_{\Omega_{\varepsilon}^T}\Xi \, \partial_s \varphi_\varepsilon(U_N + g, x, t) \,
\nabla_{\overline{\xi}_1}Y(\overline{\xi}_1)|_{\overline{\xi}_1=\frac{\overline{x}_1}{\varepsilon}}
\cdot \nabla_{\overline{x}_1}U_N  \,  dx dt.
 \end{align}
Thanks to assumption  ${\Cone}$ and estimate \eqref{estim+2} for $\{U_N\}_{N\in \Bbb N},$ again
 with the help of Lemma 2.3 \cite[Chapt. II, \S2]{Lad_Sol_Ura_1968}
we justify the $L^2(\Omega_\varepsilon^T)$-convergence of the sequences
$$
\{\varphi_\varepsilon(U_N + g, x, t)\}_{N\in\Bbb N}, \ \ \{\partial_s\varphi_\varepsilon(U_N + g, x, t)\}_{N\in\Bbb N}, \ \
\{\partial_{\xi_i}\varphi(U_N + g, x_1,\overline{\xi}_1,t)\}_{N\in\Bbb N}.
$$
 Therefore, going to the limit in \eqref{limit+2}, we get
 \begin{align*}
    \lim_{N\to +\infty} \varepsilon  \int_{\Gamma_\varepsilon^T} \Phi_\varepsilon(U_N + g, x, t) \, \Xi \, d\sigma_x dt
      \  = \ & \ \frac{\upharpoonleft\!\!\partial\varpi\!\!\upharpoonright_1}{\upharpoonleft\!\!\varpi\!\!\upharpoonright_2} \int_{\Omega_{\varepsilon}^T}\varphi_\varepsilon(\mathrm{w}_\varepsilon + g, x, t) \, \Xi \, dxdt
      \\ & \
      + \varepsilon \int_{\Omega_{\varepsilon}^T} \nabla_{\overline{\xi}_1}Y(\overline{\xi}_1)|_{\overline{\xi}_1=\frac{\overline{x}_1}{\varepsilon}}
\cdot \nabla_{\overline{x}_1}\big[\varphi_\varepsilon(\mathrm{w}_\varepsilon + g, x, t) \, \Xi \big] \, dxdt
\\
\ \stackrel{\eqref{identity}}{=} & \ \varepsilon  \int_{\Gamma_\varepsilon^T} \Phi_\varepsilon(\mathrm{w}_\varepsilon, x, t) \, \Xi \, d\sigma_x dt .
 \end{align*}

 Thus, $\mathrm{w}_\varepsilon$ is a weak solution to problem \eqref{probl+red}.

 \smallskip

 {\bf 3.} Next, we use the approach of the proof of Theorem 2.1 \cite[Chapt. V, \S2]{Lad_Sol_Ura_1968}, which says that every weak  solution is bounded. The proof of the theorem began with the deduction of the essential inequality (2.8) needed for  subsequent steps.
  In our scenario, a similar inequality can be proven in the same manner as inequality \eqref{ineq+1} but without the $\mu |u|^\delta \ (\delta> 2)$  on the right-hand side of this inequality. This means that
  \begin{equation}\label{est-solut}
    \esssup_{\Omega_\varepsilon^T}|\mathrm{w}_\varepsilon| \le C_0(\varepsilon)
  \end{equation}
 and the constant $C_0(\varepsilon)$ depends only on known quantities.

Assume that $u_\varepsilon$ and $\tilde{u}_\varepsilon$ are two weak bounded solutions to problem \eqref{probl}. Consequently
\begin{multline}\label{unique}
  \frac{1}{2}\int_{\Omega_\varepsilon}\big(u_\varepsilon(x,\tau) - \tilde{u}_\varepsilon(x,\tau)\big)^2 \, dx + \varepsilon \int_{\Omega^\tau_\varepsilon}|\nabla_x(u_\varepsilon - \tilde{u}_\varepsilon)|^2 dxdt
  \\
   - \int_{\Omega^\tau_\varepsilon} u_\varepsilon \, \big(v_1(u_\varepsilon,x_1,t) - v_1(\tilde{u}_\varepsilon,x_1,t)\big) \, \partial_{x_1}(u_\varepsilon - \tilde{u}_\varepsilon)\,  dxdt
 \\
   - \int_{\Omega^\tau_\varepsilon}v_1(\tilde{u}_\varepsilon,x_1,t) \, (u_\varepsilon - \tilde{u}_\varepsilon)\, \partial_{x_1}(u_\varepsilon - \tilde{u}_\varepsilon)\,  dxdt
   - \varepsilon \int_{\Omega^\tau_\varepsilon} (u_\varepsilon - \tilde{u}_\varepsilon)  \, \overline{V}_\varepsilon \cdot \nabla_{\overline{x}_1}(u_\varepsilon - \tilde{u}_\varepsilon)\,  dxdt
   \\
   + \varepsilon \int_{\Gamma^\tau_\varepsilon} \big(\varphi_\varepsilon(u_\varepsilon, x, t)
     -\varphi_\varepsilon(\tilde{u}_\varepsilon, x, t)\big) \, (u_\varepsilon - \tilde{u}_\varepsilon) \,  d\sigma_x dt =0
\end{multline}
for any $\tau\in [0,T].$ Using the boundedness of the solutions, the smoothness of the functions $v_1$ and $\varphi_\varepsilon$, Cauchy's inequality with a suitable $\delta>0 $ and \eqref{ineq1}, we derive from \eqref{unique} that for all $\tau \in [0,T]$
\begin{equation*}
 \int_{\Omega_\varepsilon}\big(u_\varepsilon(x,\tau) - \tilde{u}_\varepsilon(x,\tau)\big)^2 \, dx
  \le C_0(\varepsilon)\, \int_{0}^{\tau}\int_{\Omega_\varepsilon}\big(u_\varepsilon(x,\tau) - \tilde{u}_\varepsilon(x,\tau)\big)^2 \, dx \, dt,
\end{equation*}
whence by Gronwall's lemma we obtain  $u_\varepsilon = \tilde{u}_\varepsilon$ a.e. in $\Omega^T_\varepsilon.$

\medskip

 {\bf 4.} At this point, we show that due to our assumptions, the solution is classical.
  In \eqref{int-identity} we take a test function $\eta \in \mathcal{W}_\varepsilon$ with $\partial_t \eta \in L^2(\Omega^T_\varepsilon)$ and apply identity \eqref{identity} to  the integral over~$\Gamma_\varepsilon.$ Then integrating over interval $(t_0, t) \subset [0,T]$ and considering \eqref{phi_cond}, we get the inequality
  \begin{equation}\label{int-inequality}
     \int\limits_{\Omega_\varepsilon} \mathrm{w}_\varepsilon(x,\tau) \, \eta(x, \tau)\, dx \big|_{t_0}^{t}
       +  \int\limits_{t_0}^{t}\int\limits_{\Omega_\varepsilon}\Big[- \mathrm{w}_\varepsilon \, \partial_t\eta +  \sum_{i=1}^{3} \widetilde{a}_i\big(x,t,\mathrm{w}_\varepsilon, \partial_{x_i}\mathrm{w}_\varepsilon\big) \partial_{x_i}\eta \Big]\, dxdt
        \le   \int\limits_{t_0}^{t}\int\limits_{\Omega_\varepsilon}\big(\varepsilon |\nabla_{\overline{\xi}_1}\mathrm{w}_\varepsilon|^2  + C_1\big) |\eta | \, dx dt,
\end{equation}
where $ \widetilde{a}_1 =  {a}_1,$ $ \widetilde{a}_i =  {a}_i + \varepsilon\, \varphi_\varepsilon(\mathrm{w}_\varepsilon + g, x, t)\, \partial_{\xi_i} Y(\overline{\xi}_1)|_{\overline{\xi}_1=\frac{\overline{x}_1}{\varepsilon}}$ for $i\in \{2, 3\}$ (see \eqref{coeff+1} and \eqref{coeff+2}).

This inequality is used as the basis for all subsequent reasoning in the proofs of Theorem 1.1 (see Remarks 1.2 and 10.1 from \cite[Chapt. V]{Lad_Sol_Ura_1968}). Therefore, solution $\mathrm{w}_\varepsilon$ belongs to the H\"older space $H^{\alpha, \frac{\alpha}{2}}(\Omega_\varepsilon^T)$ with some $\alpha\in (0,1)$ (this is the set of functions belonging to
$H^{\alpha, \frac{\alpha}{2}}(\overline{Q})$ for any closed subdomain $\overline{Q} \subset \Omega_\varepsilon^T).$

Due to condition ${\bf A1},$ ${\bf A2},$ ${\Ctwo}$ and \eqref{est-solut}, it easy to verify all conditions of Theorem 3.1 \cite[Chapt. V]{Lad_Sol_Ura_1968}) for the coefficients (see \eqref{coeff+1} and \eqref{coeff+1}). Since $|\nabla_x \mathrm{w}_\varepsilon(x,0)|= 0,$ this theorem states that for any subdomain $Q$ of $\Omega_\varepsilon^T$ separated from $\partial \Omega_\varepsilon^T:= \partial \Omega_\varepsilon \times [0,T)$ by a positive distance $d$ the solution $ \mathrm{w}_\varepsilon \in W^{2,1}_2(Q)$ and the H\"older norm of its gradient is estimated in terms of known quantities.

The differential equation in problem \eqref{probl+red}  has only one coefficient that nonlinearly depends on $\mathrm{w}_\varepsilon$ and this equation can be considered as linear  with the coefficient  $\hat{v}_1(x,t):= v_1(\mathrm{w}_\varepsilon(x,t), x_1, t)$ that belongs to  $H^{\alpha, \frac{\alpha}{2}}(\Omega_\varepsilon).$ According to the results of \cite[Chapt. III and IV]{Lad_Sol_Ura_1968} (see e.g. Theorem 12.1 [Chapt. III]) on linear equations with smooth coefficients, we can assert that $\mathrm{w}_\varepsilon \in  H^{2+\alpha, 1+ \frac{\alpha}{2}}(\Omega_\varepsilon^T).$

To obtain  results on the  smoothness of the solution $\mathrm{w}_\varepsilon$  including the open bases
$\Upsilon_\varepsilon(0)$ and $\Upsilon_\varepsilon(\ell)$
of the cylinder $\Omega_\varepsilon,$ we rewrite the differential equation of  problem \eqref{probl+red} as follows
 \begin{equation}\label{new-dif-equa}
   \partial_t \mathrm{w}_\varepsilon  - \varepsilon \Delta_x \mathrm{w}_\varepsilon + A(x,t,\mathrm{w}_\varepsilon,\nabla_x \mathrm{w}_\varepsilon) = 0,
 \end{equation}
 where
 \begin{align*}
 A(x,t,u,p) =  & \  \big[v_1(u + g(x_1,t), x_1,t)
  + (u + g(x_1,t))\, \partial_s v_1(u + g(x_1,t),x_1,t) \big]
   \big(p_1 + \tfrac{1}{\ell} q_\ell(t)\big)
 \\
  & + (u + g(x_1,t))\, \partial_{x_1}v_1(s,x_1,t)|_{s=u + g(x_1,t)}
  \\
  & +  \varepsilon \overline{p}_1 \cdot \overline{V}_\varepsilon(x,t) + (u + g(x_1,t)) \, \big(\nabla_{\overline{\xi}_1} \boldsymbol{\cdot} \overline{V}(x_1,\overline{\xi}_1,t)\big)|_{\overline{\xi}_1=\frac{\overline{x}_1}{\varepsilon}}
    + \partial_t g(x_1,t),
 \end{align*}
 where $p=(p_1,p_2,p_3),$ $\overline{p}_1=(p_2,p_3).$ Due to conditions  ${\bf A1},$ ${\bf A2}$ and ${\Ctwo}$ it is easy to check that
 $$
| A(x,t,u,0)\, u| \le  C_1 \, u^2 + C_2 \quad \text{for all} \ \ (x,t) \in \overline{\Omega_\varepsilon^T}, \ \ u \in \Bbb R.
 $$
Also  for $(x,t) \in \overline{\Omega_\varepsilon^T},$ $|u| \le C_0$ $(C_0$ is taken from \eqref{est-solut}) and arbitrary $p\in \Bbb R^3$  the functions \eqref{coeff+1} and \eqref{coeff+2} satisfy the inequality
$$
  \sum_{i=1}^{3}\big(|a_i| + \partial_u a_i|\big) \, (1+ |p|) + \sum_{i,j=1}^{3} |\partial_{x_j} a_i|
          \le (\varepsilon |p| + C_1) \, ( 1 + |p|) + C_2 \le C_3 (1+ |p|)^2.
$$
Thus conditions a) - d) of Theorem 6.1 \cite[Chapt. V]{Lad_Sol_Ura_1968}) are satisfied.

Now take any convex subdomain $\tilde{\Omega}$ of $\Omega_\varepsilon$ with a smooth boundary intersecting, e.g., $\Upsilon_\varepsilon(0)$, but not $\overline{\Gamma_\varepsilon}.$
Applying this theorem to the equation for the function $\mathrm{w}_\varepsilon(x,t) \, \zeta(x),$ where $\zeta$ is a suitable smooth cut-off function in  $\overline{\Omega_\varepsilon}$ that is equal 1 in $\tilde{\Omega}$ and whose  support does not intersect  $\overline{\Gamma_\varepsilon},$ with the zero Dirichlet boundary condition and initial one, we conclude that $\mathrm{w}_\varepsilon \in C^{2,1}(\Omega_\varepsilon^T\cup \Upsilon^T_\varepsilon(0)),$ where
 $\Upsilon^T_\varepsilon(0):= \Upsilon_\varepsilon(0) \times (0,T).$

The boundary condition on the lateral surface  $\Gamma_\varepsilon^T$ we rewrite as follows
$\partial_{\overline{\nu}_\varepsilon} \mathrm{w}_\varepsilon + \Psi(x,t,\mathrm{w}_\varepsilon) =0,$ where
$$
\Psi(x,t,u) :=   \varphi_\varepsilon\big(u + g(x_1,t), x, t\big) -  \big(u + g(x_1,t)\big) \, \overline{V}_\varepsilon(x,t) \cdot \overline{\nu}_\varepsilon.
$$
Then, using ${\bf A1},$ ${\bf A2},$ ${\Cone}$ and ${\Ctwo},$ we verify that for any for arbitrary $u\in \Bbb R$ and $p\in \Bbb R^3$
\begin{gather*}
  |u\, A(x,t,u,p)| \le C_0 |p|^2 + C_1 u^2 + C_2 \quad \text{for all} \ \ (x,t) \in \overline{\Omega_\varepsilon^T},
   \\
  |u\, \Psi(x,t,u)| \le C_3 u^2 + C_4 \quad \text{for all} \ \ (x,t) \in \overline{\Gamma_\varepsilon^T};
\end{gather*}
and for all $(x,t) \in \overline{\Omega_\varepsilon^T},$ $|u| \le C_0$ $(C_0$ is taken from \eqref{est-solut}) and  $p\in \Bbb R^3$
\begin{gather*}
|A(x,t,u,p)| \le C_0 (1 + |p|^2),
\\
  |\partial_{p_i} A|\, (1+|p|) + |\partial_u A| + |\partial_t A| \le C_1 (1 + |p|^2),
   \\
 |\Psi(x,t,u)| + |\partial_{u}\Psi| + |\partial_{x_i}\Psi| +|\partial_{t}\Psi| + |\partial^2_{uu}\Psi| + |\partial^2_{ux}\Psi| + |\partial^2_{ut}\Psi| \le C_3.
\end{gather*}
These inequalities form the main conditions of Theorem 7.4 \cite[Chapt. V]{Lad_Sol_Ura_1968}.

Now let's take any part $\tilde{\Gamma}_{d_1,d_2} := \partial\Omega_\varepsilon \cap \{ x\colon \ 0< d_1\le x_1\le d_2 <\ell \}$ of $\Gamma_\varepsilon$ and denote by $\tilde{\Omega}$ an arbitrary subdomain of $\Omega_\varepsilon$ that is adjacent to $\tilde{\Gamma}_{d_1,d_2}$ and its closure does not intersect the bases $\overline{\Upsilon_\varepsilon(0)}$  and $\overline{\Upsilon_\varepsilon(\ell)};$ $\tilde{\Omega}^T:= \tilde{\Omega}\times (0,T).$
Applying Theorem 7.4 \cite[Chapt. V]{Lad_Sol_Ura_1968}  to the equation for the function $\mathrm{w}_\varepsilon(x,t) \, \zeta(x),$ where $\zeta$ is a suitable smooth cut-off function in  $\overline{\Omega_\varepsilon}$ that is equal 1 in $\tilde{\Omega}$ and whose  support does not intersect $\overline{\Upsilon_\varepsilon(0)}$  and $\overline{\Upsilon_\varepsilon(\ell)},$ with the corresponding Neumann condition and zero initial one,  we conclude that $\mathrm{w}_\varepsilon \in C^{2,1}(\Omega_\varepsilon^T\cup \Gamma_\varepsilon^T).$
Based on Theorem 6. 4 \cite[Chapt. V]{Lad_Sol_Ura_1968} and Theorem 4.4 \cite[Chapt. VI]{Lad_Sol_Ura_1968} and comments to them, it can be argued  that the solution $\mathrm{w}_\varepsilon$ is continuous in $\overline{\Omega}_\varepsilon \times [0,T].$
\end{proof}


\section{Formal asymptotic analysis}\label{Sec:4 Formal analysis}

We propose the following  asymptotic ansatz for the solution $u_\varepsilon$:
\begin{equation}\label{regul}
   \mathcal{U}_\varepsilon(x,t) :=  w_0(x_1,t) + \varepsilon\Big(w_1(x_1,t)  + u_1\Big(x_1, \dfrac{\overline{x}_1}{\varepsilon}, t \Big)  \Big)  + \varepsilon^2 u_2\Big(x_1, \dfrac{\overline{x}_1}{\varepsilon}, t \Big).
 \end{equation}

 Substituting $\mathcal{U}_\varepsilon$ in the differential equation of  problem \eqref{probl} and in the boundary condition at $\Gamma_\varepsilon$, using Taylor’s formula both for $v_1$ and for $\varphi_\varepsilon,$ collecting terms of the same powers of $\varepsilon$ and equating these sums  to zero, we get   Neumann problems for the coefficients $u_1$ and $u_2$
 in the rescale cross-section $\Upsilon_1 := \big\{ \overline{\xi}_1= (\xi_2, \xi_3) \in\Bbb{R}^2 \colon \  \overline{\xi}_1 \in \varpi_1 \big\}$ of the thin cylinder $\Omega_\varepsilon,$  namely
\begin{align*}
  \Delta_{\bar{\xi}_1}u_1(x_1, \bar{\xi}_1, t)
    = & \ \partial_t {w}_0(x_1,t) +  \partial_{x_1}\!\big(v_1({w}_0, x_1,t) \, {w}_0(x_1,t) \big)
     \\
     &+ w_0(x_1, t) \, \nabla_{\bar{\xi}_1} \boldsymbol{\cdot} \overline{V}(x_1, \bar{\xi}_1,t), \hspace{2.5cm} \overline{\xi}_1\in\Upsilon_1 \notag
    \\
\partial_{\vec{\nu}_{\bar{\xi}_1}} u_1(x_1, \bar{\xi}_1, t) =& \
 w_{0}(x_1,t) \, \overline{V}(x_1, \bar{\xi}_1,t) \boldsymbol{\cdot} \vec{\nu}_{\xi_1}
  - \varphi(w_{0}, x_1, \bar{\xi}_1, t), \quad \bar{\xi}_1 \in \partial\Upsilon_1,
  \\
 &  \langle u_1(x_1,\cdot, t) \rangle_{\Upsilon_1} =  0, \notag
    \end{align*}
(the first two relations were obtained equating coefficients at $\varepsilon^0);$  and at $\varepsilon^1$ we have
\begin{align*}
  \Delta_{\bar{\xi}_1}u_2
    = & \ \partial_t {w}_{1} + \partial_t {u}_1 +  \partial_{x_1}\!\big(\big(v_1({w}_0, x_1,t) +    \partial_s v_1({w}_0, x_1,t)\, {w}_0 \big) \big({w}_1 + {u}_1\big)\big) \notag
     \\
     &\ + {w}_1 \, \nabla_{\bar{\xi}_1} \boldsymbol{\cdot} \overline{V} + \nabla_{\bar{\xi}_1} \boldsymbol{\cdot}( {u}_1 \, \overline{V})
     - \partial^2_{x_1 x_1}{w}_0, \hspace{1.4cm} \bar{\xi}_1\in\Upsilon_1,  \notag
    \\
\partial_{\vec{\nu}_{\bar{\xi}_1}} u_2 =& \
 ({w}_1 + {u}_1) \, \overline{V} \boldsymbol{\cdot} \vec{\nu}_{\xi_1}
  - ({w}_1 + {u}_1) \, \partial_s \varphi(w_{0}, x_1, \bar{\xi}_1, t), \quad \bar{\xi}_1 \in \partial\Upsilon_1,
  \\
 &  \langle u_2(x_1,\cdot, t) \rangle_{\Upsilon_1} =  0, \notag
    \end{align*}
where  $\bar{\xi}_1= (\xi_2, \xi_3) :
= \big(\frac{x_2}{\varepsilon}, \frac{x_3}{\varepsilon}\big),$ $\Delta_{\bar{\xi}_1} := \frac{\partial^2}{\partial \xi_2^2} +  \frac{\partial^2}{\partial \xi_3^2},$  \  $ \partial_s v(s,x_1,t) := \frac{\partial v(s,x_1,t)}{\partial s }$,
\begin{equation}\label{overline_V}
  \overline{V}(x_1, \bar{\xi}_1,t) := \left( v_2(x_1, \bar{\xi}_1,t), \, v_3(x_1, \bar{\xi}_1,t) \right),
\end{equation}
 $\vec{\nu}_{\bar{\xi}_1} = \big(\nu_2(\xi_2,\xi_3),  \nu_3(\xi_2,\xi_3)\big)$  is the  outward unit normal to the boundary of~ $\Upsilon_1,$ $\partial_{\vec{\nu}_{\bar{\xi}_1}} $ is the derivative along $\vec{\nu}_{\bar{\xi}_1},$
 the last relations in these problems are added for the  uniqueness and
$$
  \langle u(x_1, \, \cdot \, ,  t ) \rangle_{\Upsilon_1} :=  \int_{\Upsilon_1} u(x_1, \bar{\xi}_1, t) \, d\bar{\xi}_1,
$$
the variables $x_1\in  [0, \ell]$ and $t \in [0, T]$ are  considered as parameters.

\begin{remark}\label{Rem_3-3}
Note that after the substitution of $\mathcal{U}_\varepsilon$, subsequent terms both in the differential equation and
in the boundary condition are of order $\mathcal{O}(\varepsilon^{2})$.
\end{remark}

By writing down the solvability condition for each problem, we derive the differential equations for the coefficients $w_0$ and $w_{1}$, namely
\begin{equation}\label{lim_0}
  \partial_t{w}_0  + \partial_{x_1}\!\big(v_1(w_0, x_1,t)\, {w}_0 \big) = - \widehat{\varphi}\big(w_0, x_1, t\big) \quad \text{in} \ \ \mathcal{I}\times (0, T),
\end{equation}
where
\begin{equation}\label{hat_phi}
  \widehat{\varphi}(w_0, x_1, t) := \frac{1}{\upharpoonleft\!\!\varpi\!\!\upharpoonright_2} \int_{\partial \Upsilon_1} \varphi\big(w_0(x_1,t), x_1, \bar{\xi}_1, t\big)\, d\sigma_{\bar{\xi}_1},
\end{equation}
and
\begin{equation}\label{lim_1}
  \partial_t{w}_{1} +   \partial_{x_1}\!\big(\big(v_1({w}_0,x_1,t) +    \partial_s v_1({w}_0,x_1,t)\, {w}_0 \big) \, {w}_{1} \big)
    + \partial_s \widehat{\varphi}(w_0, x_1, t) \, \, w_{1} = f_1(x_1,t)
  \end{equation}
in $\mathcal{I} \times (0, T),$ where
\begin{equation}\label{fun_1}
  f_1(x_1,t) :=  \partial^2_{x_1 x_1}{w}_0(x_1,t)   -
\frac{1}{\upharpoonleft\!\!\varpi\!\!\upharpoonright_2} \int_{\partial \Upsilon_1} \partial_s \varphi(w_0, x_1, \bar{\xi}_1, t)\,
u_{1}(x_1, \bar{\xi}_1, t) \, d\sigma_{\bar{\xi}_1}.
\end{equation}

Note that equations \eqref{lim_0} and \eqref{lim_1} are hyperbolic differentiation equations of the first order.
The first one is quasilinear, while the other is linear.


\subsection{Limit problem} Equation \eqref{lim_0}, equipped with the  boundary condition at $x_1=0$ and the initial condition, forms the \textit{limit problem}
\begin{equation}\label{limit_prob}
 \left\{\begin{array}{rcl}
\partial_t{w}_0  + \partial_{x_1}\!\left(v_1({w}_0, x_1,t)\, {w}_0 \right) & =& - \widehat{\varphi}\big({w}_0, x_1, t\big) \ \ \ \text{in} \ \ (0, \ell)\times (0, T),
    \\[2mm]
    w_0(0,  t) =  0, \ \  t \in [0, T], & &
    w_0(x_1,  0) =  0, \ \  x_1  \in [0, \ell],
  \end{array}\right.
\end{equation}
for  problem \eqref{probl}. It  can be rewritten in the Riemann invariant form
\begin{equation}\label{limit_prob+Red}
 \left\{\begin{array}{rcl}
\partial_t{w}_0(x_1,t)  +  \Lambda(w_0, x_1,t)\,  \partial_{x_1}{w}_0(x_1,t) & =& F({w}_0, x_1, t)\ \  \text{in}  \  (0, \ell)\times (0, T),
\\[2mm]
    w_0(0,  t) =  0, \ \  t \in [0, T], & &
    w_0(x_1,  0) =  0, \ \  x_1  \in [0, \ell],
  \end{array}\right.
\end{equation}
where
\begin{gather}\label{L}
  \Lambda(s, x_1,t) :=  v_1(s,x_1,t) + s \, \partial_s v_1(s, x_1,t),
  \\
    F(s,x_1,t) := - \widehat{\varphi}(s, x_1, t)  - s \, \partial_{x_1} v_1(s, x_1,t) . \label{L+}
\end{gather}

Taking into account the assumptions on smoothness of the coefficients from Section~\ref{Sec:Statement},  to satisfy the conditions of Theorem 3 \cite{Mysh-Fili_1981} it is necessary to additionally  impose the following condition:
\begin{equation}\label{add-cond-1}
  v_1(s,0,t) + s \, \partial_s v_1(s,0,t) > 0 \quad \text{for all} \ \ (s,t) \in \Bbb R \times [0,T].
\end{equation}

Let $\mathfrak{L}$ be a subset of the space $C([0,\ell]\times[0,T])$  such that each function from $\mathfrak{L}$  satisfies the Lipschitz condition. Then for any function $\Phi \in \mathfrak{L}$  there is a solution $\eta(t; y_0, t_0, \Phi)$ to
the Cauchy problem
\begin{equation}\label{Cauchy}
\frac{dy}{dt} = \Lambda(\Phi(y,t)), \quad y(t_0)= y_0 \quad \big(\,  (y_0,t_0) \in [0,\ell]\times[0,T]\, \big).
\end{equation}
In addition, the solution $\eta(t; y_0, t_0, \Phi)$
continues in the direction of decreasing $t$  and eventually reaches  the left or bottom side  of the rectangle $[0,\ell]\times[0,T].$
Let $\theta(y_0,t_0,\Phi)$ denote the smallest value of the argument~$t$ for which this solution is defined in the rectangle. Obviously,
$\theta(y_0,t_0,\Phi)\ge 0,$ $\theta(0,t_0,\Phi) =t_0,$  and if $\theta(y_0,t_0,\Phi)> 0,$ then $\eta(\theta(y_0,t_0,\Phi); y_0, t_0, \Phi) =0.$

For a continuously differentiable function $w_0 \colon  [0,\ell]\times[0,T] \mapsto \Bbb R,$ problem \eqref{limit_prob+Red} is equivalent to the integral equation
\begin{equation}\label{integral_equation}
w_0(x_1,t) = \int_{\theta(x_1,t,w_0)}^{t} F\big(w_0(\eta(\tau; x_1, t, w_0), \tau), \,  \eta(\tau; x_1, t, w_0), \, \tau\big) \, d\tau .
\end{equation}
Theorem 3 \cite{Mysh-Fili_1981} states that  there is a number $T_1 \in (0, T]$ such that
the integral equation \eqref{integral_equation} has  a unique continuous solution in $[0,\ell]\times[0,T_1]$ that satisfies the Lipschitz condition. This solution is called a generalized local solution to  problem  \eqref{limit_prob+Red}.

The interpretation of problem \eqref{probl} suggests that the solution is a concentration of a transported substance. Therefore, a condition is required to ensure that $w_0$  is nonnegative.
It is clear  from \eqref{integral_equation}  that if $F$ is nonnegative, then the generalized solution  is also nonnegative.
Therefore, from the beginning, we can look for a solution in the cone $\mathfrak{L}_{+} \subset \mathfrak{L}$ of  nonnegative functions,
and assume that
\begin{equation}\label{add-cond-2}
  \widehat{\varphi}(s, x_1, t) \le 0 \ \ \text{and} \ \  \partial_{x_1} v_1(s, x_1,t) \le 0\ \  \text{for all} \ \ (s, x_1, t) \in \Bbb R_+\times[0,\ell]\times[0,T].
\end{equation}

For this solution to be classical, as shown in \cite[\S 6, Remark 5]{Mysh-Fili_1981}, it is necessary that the functions $\Lambda$ and $F$ have continuous derivatives in $s$ and $x_1$,  these derivatives satisfy the local Lipschitz condition for the indicated variables,  and the second
matching condition holds for  problem~\eqref{limit_prob+Red}, i.e.  $F(0, 0, 0) = 0.$  Thanks to  ${\bf A2},$   ${\Cone}$ and
${\Cthree}$ $(\varphi|_{t=0}= 0),$ all these conditions are satisfied.

In addition, based on assumption $\varphi|_{t=0}= 0$ we deduce from \eqref{limit_prob+Red} that
 first $\partial_x  w_0|_{t=0} = 0$  and then  $\partial_t  w_0|_{t=0} = 0$ for all $x_1 \in [0, \ell]$.
Thus, the following statement holds.

\begin{proposition}\label{prop_1}
  Under assumptions ${\bf A2},$   ${\Cone},$ ${\Cthree},$  \eqref{add-cond-1}  and \eqref{add-cond-2},
  problem~\eqref{limit_prob}
possesses a unique classical solution in $[0,\ell]\times [0,T_1].$  Moreover, it is  nonnegative,
$\partial_t  w_0|_{t=0} = 0$ and $\partial_x  w_0|_{t=0} = 0$ for all $x_1 \in [0, \ell]$.
\end{proposition}

To determine  the  term $w_1,$  the existence of the continuous second derivative $\partial^2_{x_1 x_1} w_0$ is necessary (see~\eqref{fun_1}). Thus, the next step is to justify why this is the case.

The characteristics $x_1 = \eta_0(t) := \eta(t; 0, 0, w_0), \ t \in [0,T_1],$ divides the rectangle $[0, \ell]\times [0,T_1]$
into two regions
$$
\mathcal{R}_1 := \{(x_1, t)\colon  \ \eta_0(t) <  x_1 < \ell, \ \ t\in (0,T_1)\}  \ \  \text{and} \ \  \mathcal{R}_2:=
\{[0, \ell]\times [0,T_1]\}\setminus \overline{\mathcal{R}_1} .
$$
In $\mathcal{R}_1$,  the integral equation \eqref{integral_equation} is as follows
\begin{equation*}
w_0(x_1,t) = \int_{0}^{t} F\big(w_0(\eta(\tau; x_1, t, w_0), \tau), \,  \eta(\tau; x_1, t, w_0), \, \tau\big) \, d\tau ,
\end{equation*}
and it is equivalent to the Cauchy problem
\begin{equation}\label{limit_prob+Red+ext}
 \left\{\begin{array}{rcll}
\partial_t{w}_0(x_1,t)  +  \Lambda(w_0, x_1, t)\,  \partial_{x_1}{w}_0(x_1,t) & =& F({w}_0, x_1, t), & (x_1,t) \in \ \mathcal{R}_1,
    \\[2mm]
      w_0(x_1,  0) &=&  0,&  x_1  \in [0, \ell].
  \end{array}\right.
\end{equation}

The domain $\mathcal{R}_1$ is a domain of determinacy (see \cite{Friedrichs1948}), i.e., every point
within it can be accessed via a characteristic curve originating from the  interval $(0, \ell)$.
By  conditions ${\bf A2}$ and ${\Cone},$ the function $F$ vanishes for $x_1 \in [0, \delta_1].$ This means that $w_0$ vanishes in
$\{(x_1, t)\colon  \ \eta_0(t) \le  x_1 \le \eta_0(t) + \delta_1, \ \ t\in [0,T_1]\}.$
For the solution $w_0$ we take numbers
$$
C > \max_{(x_1,t)\in [0,\ell]\times[0,T_1]}|w_0(x_1,t)| \ \  \text{and} \ \  K >   \max_{(s,x_1,t)\in [0,C]\times[0,\ell]\times[0,T_1]}|\Lambda(s,x_1,t)|.
$$
Then, based on Theorem 7.2  in \cite{Friedrichs1948} and the remark before it, there exists $\alpha > 0$ depending on the constants
$C$, $K$ and constants bounded the second derivatives of $\Lambda$ and $F$ with respect $s$ and $x_1$ in  $[0, C]\times [0,\ell]\times[0,T_1]$ such that the  continuous derivative $\partial^2_{x_1 x_1} w_0$ exists in
$\{(x_1, t)\colon  \ \eta_0(t) \le  x_1 \le \ell, \ \ t\in [0,\alpha]\}.$
If we now apply the same reasoning, taking as the initial moment $t = \alpha$ and the initial condition $w_0|_{t=\alpha} = w_0(x_1, \alpha)$, and then repeating this process a finite number of times, we find that there is a second continuous  derivative $\partial^2_{x_1 x_1} w_0$ in
$\overline{\mathcal{R}_1}.$ In addition, it follows from \cite[Theorem 7.2]{Friedrichs1948} that $\psi := \partial^2_{x_1 x_1} w_0$  is a generalized solution to the problem
\begin{equation}\label{limit_prob+second_deriv}
 \left\{\begin{array}{rcl}
 \partial_t\psi(x_1,t)  +  \Lambda(w_0, x_1,t)\,  \partial_{x_1}\psi(x_1,t) & =& F_2(w_0, \partial_{x_1}w_0, \psi,  x_1, t) \ \ \text{in} \ \mathcal{R}_1,
    \\[2mm]
      \psi(x_1,  0) &=&  0,\quad  x_1  \in [0, \ell].
  \end{array}\right.
\end{equation}
 The differential equation in \eqref{limit_prob+second_deriv} is obtained by formally differentiating  the differential equation
in~\eqref{limit_prob+Red+ext} twice.

  In $\mathcal{R}_2$, the integral equation \eqref{integral_equation} can be rewritten in the subsequent form
\begin{equation*}
w_0(x_1,t) = \int_{0}^{x_1} \frac{F\big(w_0(y, \eta^{-1}(y; t, x_1,w_0)), \,  y, \, \eta^{-1}(y; t, x_1, w_0)\big)}
{\Lambda\big(w_0(y, \eta^{-1}(y; t, x_1,w_0)), \,  y, \, \eta^{-1}(y; t, x_1, w_0)\big)} \, dy ,
\end{equation*}
and it is equivalent to the Cauchy problem
\begin{equation}\label{limit_prob+R_2}
 \left\{\begin{array}{rcll}
\partial_{x_1}{w}_0(x_1,t) + \frac{1}{ \Lambda(w_0,x_1,t)}\partial_t{w}_0(x_1,t)    & =& \frac{F({w}_0, x_1, t)}{ \Lambda(w_0,x_1,t)}, & (x_1,t) \in \ \mathcal{R}_2,
    \\[2mm]
      w_0(0, t) &=&  0,&  t \in [0, T_1].
  \end{array}\right.
\end{equation}

As in the previous case, we show that there is a second continuous  derivative  $\partial^2_{x_1 x_1} w_0$ in
$\overline{\mathcal{R}_2}$ if
\begin{equation}\label{new-cond}
  \partial_t\varphi\big|_{t=0} = \partial^2_{tt}\varphi\big|_{t=0} =0.
\end{equation}

Now it remains to ensure the continuity of this derivative on the characteristic $x_1 = \eta_0(t), \ t \in [0,T_1].$
For this, it is necessary and sufficient that the corresponding matching condition for this characteristic be satisfied. In our case this is the equation
$$
0 = \partial_{x_1}F(0, 0, 0)   \ \Longleftrightarrow   \  \partial_{x_1}\widehat{\varphi}(0, 0, 0) = 0,
$$
which is performed thanks to  ${\Cone}$ and ${\bf A2}.$

Similarly, we show that there is a continuous derivative  $\partial^2_{t t } w_0$ in $[0, \ell]\times [0, T_1]$ if $\partial_{t}\widehat{\varphi}(0, 0, 0) = 0.$ In addition,  $\partial^2_{x_1 x_1 } w_0|_{t=0} = 0$  and $\partial^2_{t t } w_0|_{t=0} = 0.$

\begin{proposition}\label{prop_2}
Under the conditions of Proposition~\ref{prop_1} and \eqref{new-cond}, the solution $w_0$ to  problem \eqref{limit_prob} has continuous second derivatives
in $[0, \ell]\times [0, T_1],$ and $\partial^2_{x_1 x_1 }w_0|_{t=0} = 0$ and $\partial^2_{t t }w_0|_{t=0} = 0.$
\end{proposition}


\subsubsection{Determination of subsequent members}

Knowing $w_0$, we can uniquely determine $u_{1}$ as the unique solution
to the Neumann problem
\begin{equation}\label{u_1}
 \left\{\begin{array}{rcll}
 \Delta_{\bar{\xi}_1}u_1 & = & - \widehat{\varphi}\big(w_0, x_1, t\big) + w_0(x_1, t) \, \nabla_{\bar{\xi}_1} \boldsymbol{\cdot} \overline{V}(x_1, \bar{\xi}_1,t), & \overline{\xi}_1\in\Upsilon_1,
    \\[4pt]
  \partial_{\vec{\nu}_{\bar{\xi}_1}} u_1 & =&  w_{0}(x_1,t) \, \overline{V}(x_1, \bar{\xi}_1,t) \boldsymbol{\cdot} \vec{\nu}_{\xi_1}
  - \varphi(w_{0}(x_1,t), x_1, \bar{\xi}_1, t), & \bar{\xi}_1 \in \partial\Upsilon_1,
  \\[4pt]
 & & \langle u_1(x_1, \cdot,t) \rangle_{\Upsilon_1}  =  0. &
 \end{array}\right.
\end{equation}

As a result, all coefficients of the linear equation \eqref{lim_1} are determined. Equipping it with conditions
$w_1|_{x_1 =0} =0$ and $w_1|_{t =0} =0$, we get a linear mixed problem in the rectangle $(0, \ell) \times (0,T_1)$ to uniquely determine $w_1$. The existence of classical solutions to such linear mixed problems and their representations were justified in
\cite{Mel-Roh_AnalAppl-2024} (see also \cite[Appendix B]{Mel-Roh_JMAA-2024}).

The Neumann problem for  $u_2$  is as follows
\begin{equation}\label{u_2}
 \left\{\begin{array}{rcl}
 \Delta_{\bar{\xi}_1}u_2 & = & - \, \partial_s \widehat{\varphi}(w_0, x_1, t) \,  w_{1}
   -
\frac{1}{\upharpoonleft\varpi\upharpoonright_2} \int\limits_{\partial \Upsilon_1} \partial_s \varphi(w_0, x_1, \bar{\xi}_1, t)\,
u_{1}(x_1, \bar{\xi}_1, t) \, d\sigma_{\bar{\xi}_1}
   \\[3pt]
  & &+ \,\partial_t {u}_1 +  \partial_{x_1}\!\big(\big(v_1({w}_0,x_1,t) +   \partial_s v_1({w}_0,x_1,t)\, {w}_0 \big) {u}_1\big)
   \\[3pt]
  & &
     +\,  {w}_1 \, \nabla_{\bar{\xi}_1} \boldsymbol{\cdot} \overline{V} + \nabla_{\bar{\xi}_1} \boldsymbol{\cdot}( {u}_1 \, \overline{V}), \hspace{4cm}  \, \bar{\xi}_1\in\Upsilon_1,
     \\[4pt]
 \partial_{\vec{\nu}_{\bar{\xi}_1}} u_2 & =&
 ({w}_1 + {u}_1) \, \overline{V} \boldsymbol{\cdot} \vec{\nu}_{\xi_1}
  - ({w}_1 + {u}_1) \, \partial_s \varphi(w_{0}, x_1, \bar{\xi}_1, t), \ \  \hspace{1cm} \bar{\xi}_1 \in \partial\Upsilon_1,
  \\[4pt]
 & & \langle u_2(x_1, \cdot,t) \rangle_{\Upsilon_1}  =  0.
 \end{array}\right.
\end{equation}

\begin{remark}\label{Remark-3-2}
  Due to assumptions $\mathbf{A1},$ $\mathbf{A2}$ and ${\Cone}$ the coefficients $u_1$ and $u_2$ vanish if $x_1 \in [0, \delta_1]$.
  Therefore, \
$
\mathcal{U}_\varepsilon\big|_{x_1 = 0} = 0.
$

It is straightforward to verify that $u_1|_{t=0} = 0.$ As a consequence, considering $\partial^2_{x_1 x_1 }w_0|_{t=0} = 0$, we derive from
\eqref{lim_1} that $\partial_t w_1|_{t =0} =0.$

To show that  $u_2|_{t=0} = 0$, we need to establish that $\partial_t  u_1|_{t=0}=0$ and $\partial_{x_1}  u_1|_{t=0} = 0$ (see \eqref{u_2}). It follows from \eqref{u_1} that these equalities are fulfilled if
$\partial_t \varphi(0, x_1, \bar{\xi}_1, 0) = 0$ and $\partial_{x_1}\varphi(0, x_1, \bar{\xi}_1, 0) = 0$
(they hold due to the first relation in \eqref{new-cond} and the second relation in ${\Cthree}).$
 Hence,
$
\mathcal{U}_\varepsilon\big|_{t=0} = 0.
$
\end{remark}

\subsection{Boundary-layer ansatz}\label{BLPs}
The regular part \eqref{regul} of the asymptotics  satisfies the differential equation and the lateral boundary condition of problem \eqref{probl} up to terms of order $\mathcal{O}(\varepsilon^2).$
Unfortunately, $\mathcal{U}_\varepsilon$ does not satisfy the boundary condition on the right base $\Upsilon_\varepsilon(\ell)$ of the cylinder~$\Omega_\varepsilon$, leaving residuals of order $\mathcal{O}(1)$ there. To neutralize them and to satisfy the boundary condition on $\Upsilon_\varepsilon(\ell)$  in \eqref{probl}, we introduce the boundary-layer ansatz
\begin{equation}\label{prim+}
\mathcal{B}_\varepsilon(x,t)
 := \Pi_0\Big(\frac{\ell - x_1}{\varepsilon}, \frac{\bar{x}_1}{\varepsilon}, t\Big) + \varepsilon \, \Pi_1\Big(\frac{\ell - x_1}{\varepsilon}, \frac{\bar{x}_1}{\varepsilon}, t\Big)
    +
    \varepsilon^{2} \, \Pi_{2}\Big(\frac{\ell - x_1}{\varepsilon}, \frac{\bar{x}_1}{\varepsilon}, t\Big),
 \end{equation}
 which is located in a neighborhood of $\Upsilon_\varepsilon(\ell).$

Due to assumptions ${\bf A1},$ ${\bf A2}$ and ${\Cone}$ the differential equation and boundary condition on the lateral surface of the small part $\Omega_{\varepsilon, \ell -\delta_1}  :=\Omega_\varepsilon \cap \{x\colon x_1 \in (\ell -\delta_1, \ell)\}$ of the cylinder $\Omega_\varepsilon$ are as follows
$$
\partial_t u_\varepsilon  -  \varepsilon\, \Delta_{x} u_\varepsilon + \mathrm{v}_1(t) \, \partial_{x_1}\! u_\varepsilon = 0 \quad \text{and} \quad
-   \partial_{\overline{\nu}_\varepsilon} u_\varepsilon  =  \varphi_\varepsilon\big(x,t\big).
$$
Substituting  $\mathcal{B}_\varepsilon$  into problem \eqref{probl} in
this neighborhood, considering that $\mathcal{U}_\varepsilon$ almost satisfies the lateral boundary condition and collecting  coefficients at the same powers of $\varepsilon$, we get the  problems
\begin{equation}\label{prim+probl+0}
 \left\{\begin{array}{rcll}
    \Delta_\zeta \Pi_0(\zeta,t) +  \mathrm{v}_1(t) \, \partial_{\zeta_1}\Pi_0(\zeta,t)
  & =    & 0,
   & \quad \zeta\in \mathfrak{C}_+,
   \\[2mm]
  \partial_{\vec{\nu}_{\overline{\zeta}_1}} \Pi_0(\zeta,t) & =
   & 0,
   & \quad \zeta\in \partial\mathfrak{C}_+ \setminus \Upsilon_\ell,
   \\[2mm]
  \Pi_0(0, \overline{\zeta}_1,t) & =
   & \Phi_0(t),
   & \quad \overline{\zeta}_1\in\Upsilon_\ell,
   \\[2mm]
  \Pi_0(\zeta,t) & \to
   & 0,
   & \quad \zeta_1\to+\infty,
 \end{array}\right.
\end{equation}
\begin{equation}\label{prim+probl+1}
 \left\{\begin{array}{rcll}
    \Delta_\zeta \Pi_k(\zeta,t) +  \mathrm{v}_1(t) \, \partial_{\zeta_1}\Pi_k(\zeta,t)
  & =    & \partial_{t}\Pi_{k-1}(\zeta,t),
   &  \zeta\in \mathfrak{C}_+,
   \\[2mm]
  \partial_{\vec{\nu}_{\overline{\zeta}_1}} \Pi_k(\zeta,t) & =
   & 0,
   &  \zeta\in \partial\mathfrak{C}_+ \setminus \Upsilon_\ell,
   \\[2mm]
  \Pi_k(0,\overline{\zeta}_1,t) & =
   & \Phi_k(t,\overline{\zeta}_1) ,
   &  \overline{\zeta}_1\in\Upsilon_\ell,
   \\[2mm]
  \Pi_k(\zeta,t) & \to
   & 0,
   &  \zeta_1\to+\infty,\end{array}\right.
\end{equation}
for $k\in \{1, 2\}.$ Here, where $\zeta=(\zeta_1, \zeta_2, \zeta_3),$ $\zeta_1 = \frac{\ell -x_1}{\varepsilon},$ $\overline{\zeta}_1 =(\zeta_2, \zeta_3) = \frac{\overline{x}_1}{\varepsilon},$
\begin{gather*}
  \mathfrak{C}_+ :=\big\{\zeta \colon \  \overline{\zeta}_1\in \varpi, \quad \zeta_1\in(0,+\infty)\big\},
  \quad \Upsilon_\ell:=\big\{\zeta\colon  \zeta_1 = 0, \   \overline{\zeta}_1 \in \varpi\big\},
  \\
  \Phi_0(t) := q_{\ell}(t) - w_{0}(\ell,t),
\quad
\Phi_1(t,\overline{\zeta}_1) =  - w_{1}(\ell,t) - u_1(\ell,\overline{\zeta}_1,t), \quad \Phi_2(t,\overline{\zeta}_1) =  -u_2(\ell, \overline{\zeta}_1, t),
\end{gather*}
 the variable $t\in [0,T_1]$ is considered as a parameter in these problems.

\begin{proposition}
 There are solutions to problems \eqref{prim+probl+0} and \eqref{prim+probl+1} which decay exponentially to zero as $\zeta_1 \to +\infty$ uniformly in $t\in [0,T_1].$  In addition,
$\Pi_0|_{t=0} = \Pi_1|_{t=0} = 0,$ and if
\begin{equation}\label{therd_cond}
  q''_\ell(0) = 0
\end{equation}
holds, then  $\Pi_2|_{t=0} = 0.$
\end{proposition}
\begin{proof}
  {\bf 1.} Using  the  Fourier method, it is easy to find solutions to problem \eqref{prim+probl+0}, namely
\begin{equation}\label{Pi_0}
  \Pi_0(\zeta_1,t) = \Phi_0(t) \, e^{- \mathrm{v}_1(t) \, \zeta_1}.
\end{equation}
Since $\mathrm{v}_1(t) \ge \varsigma_0 > 0$ for all $t\in [0, T]$ (see ${\bf A2}$) and  $\Phi_0(t)$ are bounded with respect to $t\in [0,T_1],$
\begin{equation}\label{exp_decay}
  \Pi_0(\zeta_1,t) = \mathcal{O}\big(e^{- \varsigma_0 \zeta_1}\big) \quad \text{as} \quad \zeta_1 \to +\infty
\end{equation}
uniformly with respect to $t\in [0, T_1].$  Obviously, $\Pi_0|_{t=0} = \partial_t \Pi_0|_{t=0} =0$
(see  \eqref{match_conditions} and  the second claim of Proposition~\ref{prop_1})
 and  $\partial_t \Pi_0 = \mathcal{O}\big(e^{- \varsigma_0 \zeta_1}\big)$ as $\zeta_1 \to +\infty.$

\smallskip

{\bf 2.} We seek a solution to  problem \eqref{prim+probl+1} at $k=1$ as the sum of two functions  $\widehat{\Pi}_1$ and $\widetilde{\Pi}_1 .$ The first one depends only on $\zeta_1$ and is a solution to the problem
\begin{equation}\label{auto_eq_2+}
  \left\{
  \begin{array}{l}
  \partial^2_{\zeta_1 \zeta_1} \widehat{\Pi}_1 +  \mathrm{v}_1(t) \, \partial_{\zeta_1}\widehat{\Pi}_1  = \partial_t\Pi_0(\zeta_1, t), \quad  \zeta_1 \in (0, +\infty),
 \\[2pt]
  \Pi_1|_{\zeta_1=0} = \widehat{\Phi}_1(t):= - w_1(\ell,t).
  \end{array}
\right.
\end{equation}
Direct calculations give
\begin{equation}\label{hat-Pi_1}
\widehat{\Pi}_1 = \left(\widehat{\Phi}_1(t) +  \left(\frac{{\Phi}_0(t) \, \mathrm{v}^\prime_1(t)}{(\mathrm{v}_1(t))^2} - \frac{{\Phi}^\prime_0(t)}{\mathrm{v}_1(t)} \right) \zeta_1 + \frac{{\Phi}_0(t)\, \mathrm{v}^\prime_1(t)}{2 \mathrm{v}_1(t)} \, \zeta_1^2\right)   e^{- \mathrm{v}_1(t) \, \zeta_1}.
\end{equation}
This  representation shows that the asymptotic formula \eqref{exp_decay} also holds  for $\widehat{\Pi}_1$ and $\partial_t \widehat{\Pi}_1.$  Moreover, $\widehat{\Pi}_1|_{t=0} = 0$,  and $\partial_t \widehat{\Pi}_1|_{t=0}=0$ if
\begin{equation}\label{equa+}
  \widehat{\Phi}'_1(0)= - \partial_t w_1(\ell,t)|_{t=0} = 0 \quad \text{and} \quad \Phi''_0(0) = q''_\ell(0) - \partial^2_{tt} w_0(\ell,0) = 0.
\end{equation}
The first equality in \eqref{equa+} holds according to the second statement in Remark~\ref{Remark-3-2}, the second one is satisfied due to
\eqref{therd_cond} and Proposition~\ref{prop_2}.

The function $\widetilde{\Pi}_1 $ is a  solution to the problem
  \begin{equation*}
 \left\{\begin{array}{rcll}
    \Delta_\zeta \widetilde{\Pi}_1(\zeta,t) +  \mathrm{v}_1(t) \,  \partial_{\zeta_1}\widetilde{\Pi}_1(\zeta,t)
  & =    & 0,
   &  \zeta\in \mathfrak{C}_+,
   \\[2mm]
  \partial_{\vec{\nu}_{\overline{\zeta}_1}} \widetilde{\Pi}_1(\zeta,t) & =
   & 0,
   &  \zeta\in \partial\mathfrak{C}_+ \setminus \Upsilon_\ell,
   \\[2mm]
  \widetilde{\Pi}_1(0,\overline{\zeta}_1,t) & =
   & \widetilde{\Phi}_1(t,\overline{\zeta}_1) ,
   &  \overline{\zeta}_1\in\Upsilon_\ell,
   \\[2mm]
  \widetilde{\Pi}_1(\zeta,t) & \to
   & 0,
   &  \zeta_1\to+\infty,\end{array}\right.
\end{equation*}
where $\widetilde{\Phi}_1(t,\overline{\zeta}_1) =   - u_1(\ell, \overline{\zeta}_1,t).$
The standard substitution
\begin{equation}\label{sub_1}
  \widetilde{\Pi}_1 = e^{-\frac12 \, \mathrm{v}_1(t) \, \zeta_1} \, \mathcal{Z}
\end{equation}
reduces this problem to the following
\begin{equation}\label{prim+probl+1+}
 \left\{\begin{array}{rcll}
    \Delta_\zeta \mathcal{Z}(\zeta,t)  - \frac14 \mathrm{v}^2_1(t)\,  \mathcal{Z}(\zeta,t)
  & =    & 0,
   &  \zeta\in \mathfrak{C}_+,
   \\[2mm]
  \partial_{\vec{\nu}_{\overline{\zeta}_1}} \mathcal{Z}(\zeta,t) & =
   & 0,
   &  \zeta\in \partial\mathfrak{C}_+ \setminus \Upsilon_\ell,
   \\[2mm]
  \mathcal{Z}(0,\overline{\zeta}_1,t) & =
   & \widetilde{\Phi}_1(t,\overline{\zeta}_1) ,
   &  \overline{\zeta}_1\in\Upsilon_\ell,
   \\[2mm]
  \mathcal{Z}(\zeta,t) & \to
   & 0.
   &  \zeta_1\to+\infty,\end{array}\right.
\end{equation}
Using  the  Fourier method and recalling that $\langle u_1\rangle_{\Upsilon_1}  =  0$ (see \eqref{u_1}), the solution to problem \eqref{prim+probl+1+} can be represented as follows
\begin{equation}\label{view_solution}
\mathcal{Z}(\zeta,t)
 =   \sum\limits_{p=1}^{+\infty}a_{p}(t) \,  \Theta_p(\overline{\zeta}_1) \, X_p(\zeta_1,t),
\end{equation}
where
\begin{equation}\label{coeff-a-p}
  a_{p}(t)  =   \int_{\Upsilon_\ell} \widetilde{\Phi}_1(t,\overline{\zeta}_1)\,  \Theta_p(\overline{\zeta}_1) \, d\overline{\zeta}_1,
\end{equation}
 $\{ \Theta_p\}_{p\in\Bbb N}$ are orthonormal  in $L^2(\Upsilon_\ell)$  eigenfunctions of the Neumann spectral problem
\begin{equation*}
- \Delta_{\overline{\zeta}_1} \Theta_p =   \lambda_p \, \Theta_p \ \ \mbox{in} \ \Upsilon_\ell, \qquad
\partial_{\nu_{\overline{\zeta}_1}} \Theta_p =  0 \ \ \mbox{on} \ \partial\Upsilon_\ell,
\end{equation*}
which are orthogonal to the eigenfunction $\Theta_0 \equiv 1,$ and  $\{ X_p\}_{p\in\Bbb N}$ are solutions to the problems
\begin{equation}\label{X-solutions}
 \left\{\begin{array}{l}
    \partial^2_{\zeta_1 \zeta_1} X_p(\zeta_1,t) - \big(\frac14 \mathrm{v}^2_1(t) + \lambda_p\big) \,  X_p(\zeta_1,t) =  0,
   \quad  \zeta_1 \in (0, +\infty),
   \\[2mm]
  X_p(0,t)  = 1, \qquad  X_p(\zeta_1,t) \to 0 \ \ \text{as} \ \  \zeta_1\to+\infty,
\end{array}\right.
\end{equation}
$p\in\Bbb N,$ respectively. It is easy to calculate that
\begin{equation}\label{exp-X_p}
  X_p(\zeta_1,t) =  \exp\Big(- \zeta_1 \sqrt{\tfrac14 \mathrm{v}^2_1(t) + \lambda_p} \Big).
\end{equation}

Since  $ \lambda_p \ge \lambda_1 > 0,$ it follows from \eqref{sub_1} and \eqref{view_solution}  that
\begin{equation}\label{exp-Pi_1+}
  |\widetilde{\Pi}_1(\zeta_1, t)| \le   C_1 \, \exp\big(-\kappa_0 \zeta_1\big), \quad \text{where} \ \ \kappa_0 = \tfrac12 \, \varsigma_0 + \sqrt{\tfrac14 \varsigma^2_0 + \lambda_1}.
\end{equation}
From \eqref{coeff-a-p} and \eqref{u_1} it follows that $\max_{t\in [0,T_1]}|a_p(t)| \le C_2$ for any $p\in \Bbb N.$ Therefore, the constant $C_1$ in \eqref{exp-Pi_1+} does not depend on $t.$

 Since  $a_p(0)=0$ for all $p\in \Bbb N$ $(u_1(\ell, \overline{\zeta}_1,0) =0),$ $\widetilde{\Pi}_1|_{t=0} = 0.$
Differentiating equality \eqref{sub_1} with respect to the variable $t,$ we get the representation for $\partial_t \widetilde{\Pi}_1.$
Taking into account that $\mathcal{Z}|_{t=0}=0$ and $a'_p(0)=0$ for all $p\in \Bbb N$ (this is because $\partial_t u_1(\ell, \overline{\zeta}_1,t)|_{t=0} =0$ (see the third statement in Remark~\ref{Remark-3-2})), we conclude that $\partial_t \widetilde{\Pi}_1|_{t=0}=0.$

To demonstrate the uniform exponential decay of $\partial_t \widetilde{\Pi}_1$ as $\zeta_1 \to +\infty,$ it suffices to prove that
solutions to problems
\begin{equation}\label{Y-solutions}
 \left\{\begin{array}{l}
    \partial^2_{\zeta_1 \zeta_1} Y_p(\zeta_1,t) - \big(\frac14 \mathrm{v}^2_1(t) + \lambda_p\big) \,  Y_p(\zeta_1,t) =  f_p(\zeta_1,t),
   \quad  \zeta_1 \in (0, +\infty),
   \\[2mm]
  Y_p(0,t)  = 0, \qquad  Y_p(\zeta_1,t) \to 0 \ \ \text{as} \ \  \zeta_1\to+\infty,
\end{array}\right.
\end{equation}
$p\in\Bbb N,$ also decay exponentially uniformly with respect to $t\in[0, T_1].$
Here, $Y_p = \partial_t X_p$ and $f_p(\zeta_1,t) = \frac12 \mathrm{v}_1(t) \, \mathrm{v}^\prime_1(t)  \,  X_p(\zeta_1,t).$

Based on \eqref{exp-X_p},  the right-hand side $f_p$  of the differential equation in problem~\eqref{Y-solutions}  decreases exponentially to zero as $\zeta_1 \to +\infty$  uniformly over $t\in [0,T_1].$
For the solution  we derive the representation
\begin{equation*}
  Y_p(\zeta_1,t) = - X_p(\zeta_1,t)\int_{0}^{\zeta_1} \frac{1}{X^2_p(\mu,t)} \int_{\mu}^{+\infty} X_p(\theta,t) \, f_p(\theta,t)\, d\theta \, d\mu
\end{equation*}
from which follows directly  the uniform in $t\in [0,T]$ exponential decrease of $Y_p$ at infinity.

Thus, the statement of the lemma is valid  for the solution $\Pi_1$  to  problem \eqref{prim+probl+1} at $k=1$
and $\partial_t\Pi_1|_{t=0}=0$ if condition \eqref{therd_cond} is fulfilled.

In a similar way, we justify the existence and uniform exponential decay of the solution to problem  \eqref{prim+probl+1} at $k=2.$
Since $\partial_t\Pi_1|_{t=0}=0$ and  $u_2(\ell, \overline{\zeta}_1,0) =0$ (see the third claim in Remark~\ref{Remark-3-2}), we
have $\Pi_2|_{t=0}=0.$
\end{proof}

Thus, we have constructed the boundary-layer ansatz \eqref{prim+} which decays exponentially to zero as $\zeta_1 \to +\infty$ uniformly over $t\in [0,T_1].$ Moreover,
\begin{equation}\label{results-1}
  \mathcal{B}_\varepsilon(x,t)|_{t=0} = 0 \quad \text{and} \quad \mathcal{B}_\varepsilon(x,t)|_{x_1=\ell} + \mathcal{U}_\varepsilon(x,t)|_{x_1=\ell} = q_\ell(t).
\end{equation}

\section{Justification}\label{Sec:justification}

Now, using these ansatzes and  the smooth cut-off function
\begin{equation}\label{cut-off_functions}
\chi_\ell(x_1) =
    \left\{
    \begin{array}{ll}
        1, & \text{if} \ \ x_1 \ge \ell_3 -  \frac12 \delta_1,
    \\[4pt]
        0, & \text{if} \ \ x_1 \le \ell_3 - \delta_1,
    \end{array}
    \right.
\end{equation}
where  $\delta_1$ is defined in assumption ${\bf A2},$  we construct the approximation function
\begin{equation}\label{approx}
  \mathfrak{A}_\varepsilon(x,t) :=  \mathcal{U}_\varepsilon(x,t) + \chi_\ell(x_1) \, \mathcal{B}_\varepsilon(x,t), \quad (x,t) \in \Omega_\varepsilon^{T_1}.
\end{equation}

The summand $\chi_\ell(x_1) \, \mathcal{B}_\varepsilon$ localized in $\Omega_\varepsilon \cap \{x\colon x_1 \in (\ell -\delta_1, \ell)\}$
and
\begin{multline}\label{Res_3}
\partial_t\big(\chi_\ell \, \mathcal{B}_\varepsilon\big)  -  \varepsilon\, \Delta_{x}\big(\chi_\ell \, \mathcal{B}_\varepsilon\big) + \mathrm{v}_1(t) \, \partial_{x_1}\!\big(\chi_\ell \, \mathcal{B}_\varepsilon\big)
                     = \varepsilon^2\, \chi_\ell \, \partial_t \Pi_2
       \\
       - \varepsilon \,  \big(\chi_\ell\big)^{\prime\prime}  \sum\limits_{k=0}^{2} \varepsilon^{k} \, \Pi_k
       + 2  \big(\chi_\ell\big)'  \sum\limits_{k=0}^{2} \varepsilon^{k} \partial_{\zeta_1} \Pi_k
  +  \mathrm{v}_1 \,  \big(\chi_\ell\big)'  \sum\limits_{k=0}^{2} \varepsilon^{k} \Pi_k.
  \end{multline}
The supports of summands in the second line of this equation coincide with $\mathrm{supp}\big((\chi_\ell)'\big),$ where
the functions $\{\Pi_k, \, \partial_{\zeta_1} \Pi_k \}_{k=0}^{2}$ exponentially decay as $\varepsilon$ tends to zero. Therefore, the right-hand side of the differential equation \eqref{Res_3} has order $\mathcal{O}(\varepsilon^2)$ for $\varepsilon$ small enough.

Summing-up  calculations in Section~\ref{Sec:4 Formal analysis} (see Remark \ref{Rem_3-3}, the last statement in Remark~\ref{Remark-3-2} and \eqref{results-1})  and in this one, we get the statement.

\begin{lemma}\label{lemma-1} Assume that, in addition to the main assumptions made in Section~\ref{Sec:Statement},
assumptions \eqref{add-cond-1}, \eqref{add-cond-2}, \eqref{new-cond} and \eqref{therd_cond} are satisfied.
Then there is a positive number $\varepsilon_0$ such that for all $\varepsilon\in (0, \varepsilon_0)$  the difference between the approximation function~\eqref{approx} and the solution to  problem \eqref{probl} satisfies the following relations:
\begin{multline}\label{dif_1}
  \partial_t(\mathfrak{A}_\varepsilon - u_\varepsilon) -  \varepsilon\, \Delta_x \big(\mathfrak{A}_\varepsilon - u_\varepsilon\big) +
  \partial_{x_1}\!\Big(\mathfrak{A}_\varepsilon \, v_1(\mathfrak{A}_\varepsilon,x_1,t)\Big)
  - \partial_{x_1}\!\Big(u_\varepsilon \, v_1(u_\varepsilon,x_1,t)\Big)
  \\
   + \varepsilon\, \nabla_{\overline{x}_1} \boldsymbol{\cdot} \Big(\big(\mathfrak{A}_\varepsilon - u_\varepsilon\big)
 \, \overline{V}_\varepsilon\Big)
    =  \varepsilon^2\,  \mathcal{F}^{(0)}_\varepsilon \quad  \text{in} \ \ \Omega_{\varepsilon}^{T_1},
\end{multline}
\begin{equation}\label{bc_1}
 -    \partial_{\vec{\nu}_\varepsilon}(\mathfrak{A}_\varepsilon - u_\varepsilon) +  (\mathfrak{A}_\varepsilon - u_\varepsilon) \, \overline{V}_\varepsilon\boldsymbol{\cdot}\vec{\nu}_\varepsilon
 -   \varphi_\varepsilon\big(\mathfrak{A}_\varepsilon, x,t)  + \varphi_\varepsilon\big(u_\varepsilon, x,t)
  =  \varepsilon^2\,  \mathcal{F}^{(1)}_\varepsilon \quad  \text{on} \ \ \Gamma_{\varepsilon}^{T_1},
\end{equation}
\begin{equation}\label{bc_2}
 \mathfrak{A}_\varepsilon - u_\varepsilon = 0 \quad  \text{on} \  \ \Upsilon_{\varepsilon}^{T_1} (0) \cup
 \Upsilon_{\varepsilon}^{T_1} (\ell),
\qquad  (\mathfrak{A}_\varepsilon - u_\varepsilon)\big|_{t=0}  = 0 \quad \text{on} \ \ \Omega_{\varepsilon},
\end{equation}
where
\begin{equation}\label{Res_10}
  \sup\nolimits_{\Omega_{\varepsilon}^{T_1}} |\mathcal{F}^{(0)}_\varepsilon| +
  \sup\nolimits_{\Gamma_{\varepsilon}^{T_1}} |\mathcal{F}^{(1)}_\varepsilon| \le C_1.
\end{equation}
\end{lemma}

\begin{remark}
 In \eqref{Res_10} and  onwards, all constants in inequalities are independent of the solution $u_\varepsilon$ and  the parameter~$\varepsilon.$
\end{remark}

Denote by $\mathfrak{A}_{1,\varepsilon}$ the approximation function $\mathfrak{A}_\varepsilon$ without the terms $ \varepsilon^2 u_2$ (see \eqref{regul})  and  $\varepsilon^2 \Pi_2$ (see \eqref{prim+}).

\begin{theorem}\label{Th_1} Let assumptions of Lemma~\ref{lemma-1} be satisfied.
Then,  there are positive constants $\tilde{C}_0,$ $\tilde{C}_1$ and $\varepsilon_0$ such that for all $\varepsilon\in (0, \varepsilon_0)$
 \begin{equation}\label{max_1+}
  \max_{\overline{\Omega}_\varepsilon \times [0,T_1]} |\mathfrak{A}_{1,\varepsilon} - u_\varepsilon|  \le \tilde{C}_0 \, \varepsilon^{2},
\end{equation}
and
\begin{equation}\label{app_estimate}
\tfrac{1}{\sqrt{\upharpoonleft  \Omega_\varepsilon \upharpoonright_3}}\,{\|
\nabla_x \mathfrak{A}_{1,\varepsilon} - \nabla_x u_\varepsilon\|}_{L^2(\Omega_\varepsilon\times (0, T_1))} \le
 \tilde{C}_1\, \varepsilon.
\end{equation}
 \end{theorem}
\begin{proof}
  {\bf 1.}
Denote by  $\mathfrak{R}_\varepsilon := \mathfrak{A}_\varepsilon - u_\varepsilon.$
Thanks to assumptions ${\bf A2},$  the difference
$$
\partial_{x_1}\!\big(\mathfrak{A}_\varepsilon \, v_1(\mathfrak{A}_\varepsilon,x_1,t)\big) - \partial_{x_1}\!\big(u_\varepsilon \, v_1(u_\varepsilon,x_1,t)\big)
$$
in \eqref{dif_1} is equal to  $\mathrm{v}_1(t) \, \partial_{x_1}\! \mathfrak{R}_\varepsilon$
 in $\Omega_{\varepsilon, \ell -\delta_1}:=\Omega_\varepsilon \cap \{x\colon x_1 \in (\ell -\delta_1, \ell)\},$  and in the other part of the cylinder~$\Omega_\varepsilon,$ based on the mean value theorem, to
 \begin{equation}\label{difference+1}
   \Lambda'(\Theta_\varepsilon, x_1,t) \, \partial_{x_1}\mathcal{U}_\varepsilon \,  \mathfrak{R}_\varepsilon + \Lambda(u_\varepsilon,x_1,t)\, \partial_{x_1}\mathfrak{R}_\varepsilon,
 \end{equation}
 where the function $\Lambda$ is defined in \eqref{L}.  Taking into account the smoothness of the coefficients of the regular ansatz $\mathcal{U}_\varepsilon$ and inequalities \eqref{v_1}, we get
 \begin{equation}\label{est+1}
   \sup\nolimits_{\Omega_\varepsilon^{T_1}} | \Lambda'(\Theta_\varepsilon, x_1,t) \, \partial_{x_1}\mathcal{U}_\varepsilon|
  \le \widehat{C}_1.
 \end{equation}

 Considering \eqref{difference+1}, equation \eqref{dif_1} can be reduced with the substitution  $\Psi_\varepsilon = \mathfrak{R}_\varepsilon  \, e^{-\lambda t}$ to
\begin{multline}\label{eq_psi3}
  \partial_t\Psi_\varepsilon   -  \varepsilon\, \Delta\Psi_\varepsilon +
  \Lambda(u_\varepsilon,x_1,t)\, \partial_{x_1}\Psi_\varepsilon + \varepsilon \overline{V}_\varepsilon \cdot \nabla_{\overline{x}_1} \Psi_\varepsilon
  +
  \Big(  \Lambda'(\Theta_\varepsilon, x_1,t) \, \partial_{x_1}\mathcal{U}_\varepsilon + \nabla_{\bar{\xi}_1}\cdot \overline{V} +\lambda\Big) \Psi_\varepsilon
   \\
   =  \varepsilon^2 \, e^{-\lambda t}\,  \mathcal{F}^{(0)}_\varepsilon \quad  \text{in} \ \ \big(\Omega_{\varepsilon} \setminus
  \Omega_{\varepsilon, \ell -\delta_1}\big) \times (0, T_1),
\end{multline}
and to
 \begin{equation}\label{eq_psi2}
\partial_t \Psi_\varepsilon -  \varepsilon\, \Delta_x\Psi_\varepsilon +
   \mathrm{v}_1(t) \,\partial_{x_1}\Psi_\varepsilon  + \lambda \Psi_\varepsilon =  \varepsilon^2 \,  e^{-\lambda t}\, \mathcal{F}^{(0)}_{\varepsilon}  \quad     \text{in} \ \  \Omega_{\varepsilon, \ell -\delta_1}^{T_1}.
\end{equation}
Here  the   constant $\lambda$ will be chosen below and  $\overline{V}(x_1, \overline{\xi}_1, t)$ is defined in \eqref{overline_V}.

 The boundary condition  \eqref{bc_1} for $\Psi_\varepsilon$ can be rewritten as follows
\begin{equation}\label{eq_4}
\partial_{\vec{\nu}_\varepsilon}\Psi_\varepsilon
  = - \varepsilon^2\,   e^{-\lambda t}\, \mathcal{F}^{(1)}_\varepsilon \quad  \text{on} \ \ \big(\Gamma_{\varepsilon}\cap \{x\colon x_1 \in (\ell -\delta_1, \ell)\}\big) \times  (0,T_1)
\end{equation}
because of assumption ${\Cone},$ \ and
\begin{equation*}
   \partial_{\vec{\nu}_\varepsilon}\Psi_\varepsilon +  \Psi_\varepsilon \, \big( \partial_s \varphi_\varepsilon(\theta,x,t) - \overline{V}_\varepsilon\boldsymbol{\cdot}\vec{\nu}_\varepsilon\big)
  =  - \varepsilon^2\,   e^{-\lambda t}\, \mathcal{F}^{(1)}_\varepsilon
\end{equation*}
on $\big(\Gamma_{\varepsilon}\cap \{x\colon x_1 \in (0, \ell -\delta_1)\}\big) \times  (0,T_1).$
\smallskip

{\bf 2.}
Further, without loss of generality, we can assume that the domain $\varpi$ is a disk $\{\overline{x}_1\colon x_2^2 + x_3^2 < r_0^2\}.$
Consider the function
\begin{equation*}
  Z_\varepsilon(\overline{x}_1) = 1 + \varepsilon\, \frac{\varrho \, r_0}{2 } \left( 1 - \Big(\frac{x_2}{\varepsilon r_0}\Big)^2 - \Big(\frac{x_3}{\varepsilon r_0}\Big)^2\right), \quad x \in \overline{\Omega}_\varepsilon,
\end{equation*}
where  the constant
\begin{equation}\label{m_1}
 \varrho  := 1 +  \max_{\mathcal{X}} \big|\partial_s \varphi^{(i)}(s,x,t)\big| + \max_{\overline{\Omega}_1\times [0,T]}\big|\overline{V}\big(x, t\big)\big|.
\end{equation}
Here $\mathcal{X}$ is the domain of the function $\varphi$ (see ${\Cone})$ and $\overline{\Omega}_1\times [0,T]$ is the domain of $\overline{V}$ (see ${\bf A1}).$ In general case, we should take a function
$$
Z_\varepsilon(\overline{x}_1) = 1 + \varepsilon\, h\big(\tfrac{\overline{x}_1}{\varepsilon}\big),
$$
where $h$ is a function from $C^2(\overline{\varpi})$ such that $h|_{\partial\varpi}=0$ and $\partial_{\overline{\nu}} h|_{\partial\varpi}= - \varrho.$

It is easy to verify that
\begin{equation}\label{z_1}
  Z_\varepsilon\big|_{\Gamma_\varepsilon} =1, \qquad - \partial_{\vec{\nu}_\varepsilon} Z_\varepsilon\big|_{\Gamma_\varepsilon} = \varrho , \qquad
  1 < Z_\varepsilon \le 1  +  \tfrac{1}{2 }\, \varepsilon \, \varrho  \, r_0 \quad \text{in} \ \ \Omega_\varepsilon,
  \end{equation}
  \begin{equation}\label{z_2}
   |\nabla_{\overline{x}_1 } Z_\varepsilon|^2  = \frac{\varrho^2 \, |\overline{x}_1|^2}{\varepsilon^2 r_0^2} \le \varrho^2
   \quad \text{and} \quad
  \Delta_{\overline{x}_1 } Z_\varepsilon  =  -\varepsilon^{-1} \frac{2 \varrho}{r_0} \quad \text{in} \ \ \Omega_\varepsilon .
  \end{equation}

Now we introduce a new function $\mathcal{K}_\varepsilon = \Psi_\varepsilon \, Z_\varepsilon.$ Direct calculations show that
the function $\mathcal{K}_\varepsilon$ satisfies the differential equation
\begin{equation}\label{eq_k+1}
\partial_t \mathcal{K}_\varepsilon -  \varepsilon\, \Delta_x\mathcal{K}_\varepsilon  + \frac{2 \varepsilon}{Z_\varepsilon}\, \nabla_{\overline{x}_1} Z_\varepsilon \cdot \nabla_{\overline{x}_1}\mathcal{K}_\varepsilon +
   \mathrm{v}_1(t) \,\partial_{x_1}\mathcal{K}_\varepsilon
     +
   \Big(-\frac{2 \varrho}{r_0\, Z_\varepsilon} - \frac{2 \varepsilon\,|\nabla_{\overline{x}_1 } Z_\varepsilon|^2}{Z_\varepsilon^2} +\lambda\Big) \mathcal{K}_\varepsilon
    =  \varepsilon^2 \,  e^{-\lambda t} \, Z_\varepsilon \, \mathcal{F}^{(0)}_{\varepsilon}
\end{equation}
 in \  $\Omega_{\varepsilon, \ell -\delta_1}^{T_1},$
and
\begin{multline}\label{eq_k+2}
  \partial_t\mathcal{K}_\varepsilon   -  \varepsilon\, \Delta_x\mathcal{K}_\varepsilon + \Big( \varepsilon \overline{V}_\varepsilon + \frac{2 \varepsilon}{Z_\varepsilon}\, \nabla_{\overline{x}_1} Z_\varepsilon\Big) \cdot \nabla_{\overline{x}_1}\mathcal{K}_\varepsilon+
  \Lambda(u_\varepsilon,x_1,t)\, \partial_{x_1}\mathcal{K}_\varepsilon
  \\ +
  \Big(-\frac{2 \varrho}{r_0\, Z_\varepsilon} - \frac{2 \varepsilon\,|\nabla_{\overline{x}_1 } Z_\varepsilon|^2}{Z_\varepsilon^2} - \frac{ \varepsilon \, \overline{V}_\varepsilon \cdot \nabla_{\overline{x}_1} Z_\varepsilon}{Z_\varepsilon}
  +
  \Lambda'(\Theta_\varepsilon, x_1,t) \, \partial_{x_1}\mathcal{U}_\varepsilon + \nabla_{\bar{\xi}_1}\cdot \overline{V} +\lambda\Big) \mathcal{K}_\varepsilon
  \\
  =   \varepsilon^2   e^{-\lambda t} Z_\varepsilon \,  \mathcal{F}^{(0)}_\varepsilon \quad \text{in} \ \
  \big(\Omega_{\varepsilon} \setminus\Omega_{\varepsilon, \ell -\delta_1}\big) \times (0, T_1),
\end{multline}
and, by virtue of the first two equalities in \eqref{z_1}, the boundary condition
\begin{equation}\label{eq_4+k}
 \partial_{\vec{\nu}_\varepsilon}\mathcal{K}_\varepsilon + \varrho\, \mathcal{K}_\varepsilon
  = - \varepsilon^2\,   e^{-\lambda t}\, \mathcal{F}^{(1)}_\varepsilon \quad  \text{on} \ \ \Gamma_{\varepsilon}\cap \{x\colon x_1 \in (\ell -\delta_1, \ell)\},
\end{equation}
and
\begin{equation}\label{eq_5+k}
   \partial_{\vec{\nu}_\varepsilon}\mathcal{K}_\varepsilon +  \mathcal{K}_\varepsilon \, \big( \varrho +  \partial_s \varphi_\varepsilon(\theta,x,t) - \overline{V}_\varepsilon\boldsymbol{\cdot}\vec{\nu}_\varepsilon\big)
  =   - \varepsilon^2\,   e^{-\lambda t}\, \mathcal{F}^{(1)}_\varepsilon
\end{equation}
on $\big(\Gamma_{\varepsilon}\cap \{x\colon x_1 \in (0, \ell -\delta_1)\}\big)\times (0,T_1].$

Thanks to our  choice  of the constant $\varrho$ (see \eqref{m_1}),  in \eqref{eq_5+k} we have
\begin{equation}\label{eq_6++}
 \varrho + \partial_s \varphi_\varepsilon(\theta,x,t)  - \overline{V} \boldsymbol{\cdot}\overline{\nu}_\varepsilon  \ge 1.
\end{equation}

Now we choose  the   constant
\begin{equation}\label{e_7}
   \lambda :=  1+
    \frac{2 \varrho}{r_0} + 2 \, \varrho^2 + \varrho \max_{\overline{\Omega}_1\times [0,T]}\big|\overline{V}\big(x, t\big)\big| \notag
    + \widehat{C}_1 + \max_{(x_1,\overline{\xi}_1, t) \in[0,\ell]\times \Upsilon_1 \times [0,T_1]} \big| \nabla_{\overline{\xi}_1}\cdot \overline{V}(x_1, \overline{\xi}_1, t)\big|,
\end{equation}
 where $\widehat{C}_1$ is from inequality \eqref{est+1}. Hence,  the coefficient at $\mathcal{K}_\varepsilon$ both  in \eqref{eq_k+1} and in \eqref{eq_k+2}  is bounded from  below by $1.$

\smallskip

{\bf 3.}
First we suppose that  the positive maximum of  $\mathcal{K}_\varepsilon$ is  reached at a point $P_0 = (x^0, t_0)\in \Omega_\varepsilon \times (0, T_1].$
 Then the following relations
\begin{equation*}
  \partial_t \mathcal{K}_\varepsilon \ge 0, \qquad \nabla_x\mathcal{K}_\varepsilon = \vec{0}, \qquad - \Delta_x \mathcal{K}_\varepsilon \ge 0
\end{equation*}
are satisfied at $P_0.$ Therefore, it follows from \eqref{eq_k+1} and  \eqref{eq_k+2}  that
\begin{equation*}
   \mathcal{K}_\varepsilon\big|_{P_0}  \le   \varepsilon^{2} \,  C_2 \, \max_{\overline{\Omega}_\varepsilon \times [0,T_1]}|\mathcal{F}^{(0)}_\varepsilon|.
\end{equation*}

Similarly, the case where $\mathcal{K}_\varepsilon$ takes the smallest negative value at a point from $\Omega_\varepsilon \times (0, T_1]$ is considered. As a result,  we obtain
 \begin{equation}\label{e_8}
    \max_{\overline{\Omega}_\varepsilon \times [0,T_1]} |\mathcal{K}_\varepsilon|  \le   \varepsilon^2\, C_2 \,
    \max_{\overline{\Omega}_\varepsilon \times [0,T_1]}|\mathcal{F}^{(0)}_\varepsilon|.
 \end{equation}

Now consider the case when the function $\mathcal{K}_\varepsilon$ takes the largest positive value at a  point $P_1$ belonging to $\Gamma_\varepsilon \times (0, T_1].$ Then $\partial_{\boldsymbol{\nu}_\varepsilon}\mathcal{K}_\varepsilon\big|_{P_1}  \ge 0$
 and  from \eqref{eq_4+k} --  \eqref{eq_6++} we get
$$
 \mathcal{K}_\varepsilon(P_1) \le -  \varepsilon^{2} \big( e^{-\lambda t}\,  \mathcal{F}^{(1)}_\varepsilon\big)\big|_{P_1} \ \ \Longrightarrow \ \
 \mathcal{K}_\varepsilon(P_1) \le   \varepsilon^{2}  \max_{\overline{\Gamma}_\varepsilon \times [0,T_1]} |\mathcal{F}^{(1)}_\varepsilon|.
$$
In the case when  $\mathcal{K}_\varepsilon$ reaches its smallest negative value  at a point
$P_2 \in  \Gamma_\varepsilon\times (0, T_1)$ we should consider
the function  $- \mathcal{K}_\varepsilon$ and repeat the previous argumentations.

Thus,
\begin{equation}\label{est-res}
    \max_{\overline{\Omega}_\varepsilon \times [0,T_1]}|\mathcal{K}_\varepsilon|  \le   \varepsilon^2\, C_2 \,
    \max_{\overline{\Omega}_\varepsilon \times [0,T_1]}|\mathcal{F}^{(0)}_\varepsilon| +  \varepsilon^{2} \max_{\overline{\Gamma}_\varepsilon \times [0,T_1]} |\mathcal{F}^{(1)}_\varepsilon|  \stackrel{\eqref{Res_10}}{\le } \varepsilon^2\, C_3.
\end{equation}

Now returning to the function $\mathfrak{R}_\varepsilon$ and using \eqref{est-res}, we have
\begin{equation}\label{max_1}
   \max_{\overline{\Omega}_\varepsilon \times [0,T_1]} |\mathfrak{R}_\varepsilon|   =
  \max_{\overline{\Omega}_\varepsilon \times [0,T_1]} |e^{\lambda t} \, \tfrac{1}{Z_\varepsilon}\mathcal{K}_\varepsilon| \le \varepsilon^2\, C_3 \, e^{\lambda T}.
\end{equation}
It follows from \eqref{max_1}  that terms of order $\mathcal{O}(\varepsilon^2)$ are redundant in this estimate, and as a result
we get
\begin{multline*}
   \max_{\overline{\Omega}_\varepsilon \times [0,T_1]} |\mathfrak{A}_{1,\varepsilon} - u_\varepsilon|  \le
   \max_{\overline{\Omega}_\varepsilon \times [0,T_1]} |\mathfrak{R}_\varepsilon| +
   \max_{\overline{\Omega}_\varepsilon \times [0,T_1]} |\varepsilon^2 u_2|
      + \max_{\overline{\Omega}_\varepsilon \times [0,T_1]} |\varepsilon^2 \, \chi_\ell \, \Pi_2| = \mathcal{O}(\varepsilon^2) \quad \text{as} \ \ \varepsilon \to 0.
\end{multline*}

\smallskip

{\bf 4.} Here,  we prove  estimate \eqref{app_estimate}. We multiply the differential equation~\eqref{dif_1}  with  $\mathfrak{R}_\varepsilon$ and integrate it over $\Omega_\varepsilon^{T_1}.$  Integrating by parts and taking \eqref{bc_1}, \eqref{bc_2}, \eqref{Res_10}, \eqref{difference+1}, \eqref{est+1},    \eqref{phi_cond+1} and \eqref{max_1} into account, we deduce the inequality
\begin{equation}\label{est11}
  \varepsilon \, {\| \nabla_x \mathfrak{R}_\varepsilon}\|^2_{L^2(\Omega_\varepsilon^{T_1})} \le C_1 \varepsilon^5 + \varepsilon^2 \int_{\Omega_\varepsilon^{T_1}} |\Lambda(u_\varepsilon,x_1,t)| \, |\partial_{x_1}\mathfrak{R}_\varepsilon| \, dx.
 \end{equation}
It follows from \eqref{v_1} that the value $|\Lambda(u_\varepsilon,x_1,t)|$ is bounded by a constant $C_2.$ Then, using Cauchy's inequality with $\delta = \frac{1}{2 \varepsilon}, $ we deduce  from \eqref{est11}  the inequality
\begin{equation}\label{app_estimate+1}
\tfrac{1}{\sqrt{\upharpoonleft  \Omega_\varepsilon \upharpoonright_3}}\,{\| \nabla_x \mathfrak{R}_\varepsilon\|}_{L^2(\Omega_\varepsilon\times (0, T_1))} \le
 C_3\, \varepsilon.
\end{equation}
  Since
  $$
  \tfrac{1}{\sqrt{\upharpoonleft  \Omega_\varepsilon \upharpoonright_3}}\,{\| \nabla_x(\varepsilon^2 u_2)\|}_{L^2(\Omega_\varepsilon\times (0, T_1))} = \mathcal{O}(\varepsilon)
  \quad
 \text{and}
  \quad
  \tfrac{1}{\sqrt{\upharpoonleft  \Omega_\varepsilon \upharpoonright_3}}\,{\| \nabla_x(\varepsilon^2 \, \chi_\ell \, \Pi_2)\|}_{L^2(\Omega_\varepsilon\times (0, T_1))} = \mathcal{O}(\varepsilon)
  $$
  as $\varepsilon$ goes to zero, from \eqref{app_estimate+1} it follows  \eqref{app_estimate}.
\end{proof}

Estimate \eqref{max_1+} leads to the following inequalities.
\begin{corollary}\label{corol_2} For $\varepsilon$ small enough
  \begin{equation}\label{max_1+1}
  \max_{\overline{\Omega}_\varepsilon \times [0,T_1]} \left|u_\varepsilon(x,t) - \mathfrak{A}_{0,\varepsilon} \right|  \le C_0 \, \varepsilon,
\end{equation}
\begin{equation}\label{app_estimate+2}
 \max_{(x_1,t)\in [0, \ell]\times [0,T_1]} \left| \tfrac{1}{ \upharpoonleft \varepsilon \varpi \upharpoonright_2}\int_{\varepsilon \varpi}  u_\varepsilon(x,t)\, d\bar{x}_1 -  \mathfrak{A}_{0,\varepsilon} \right|  \le C_1 \, \varepsilon.
\end{equation}
where $\mathfrak{A}_{0,\varepsilon} :=  w_0(x_1,t) + \chi_\ell(x_1) \, \big(q_{\ell}(t) - w_{0}(\ell,t)\big) \,
  \exp\big(- \mathrm{v}_1(t) \, \tfrac{\ell -x_1}{\varepsilon}\big).$
\end{corollary}


\section{Problems with higher order P\'eclet numbers,  Conclusions and Discussions}\label{Sect-Conclusion}

{\bf 1.} Estimate \eqref{max_1+1} indicates that the solution $u_\varepsilon$ exhibits a boundary-layer structure in a vicinity of the right base of the thin cylinder $\Omega_\varepsilon.$
Estimate~\eqref{max_1+} shows that the solution is not constant with respect to the transversal variables in any  cross-section of the cylinder; the term $u_1$ (see \eqref{regul}) depends on the variables $\frac{\overline{x}_1}{\varepsilon}.$
Therefore, for both the original problem \eqref{probl} and many singularly perturbed problems, it is insufficient to demonstrate the convergence of the solution. It is imperative to construct the asymptotics of the subsequent terms, and
proving an error estimate between the constructed approximation and the solution should be a general principle for analysing the effectiveness of a proposed reduction method. This helps to ensure accuracy and reliability.

\smallskip

{\bf 2.} An interesting but unresolved case  is the scaling of the diffusion operator $\Delta_x$ with  an intensity factor of type  $\varepsilon^\beta.$ For many real processes, the diffusion coefficient is  very small. Therefore, it is reasonable to assume $\beta > 1.$ Currently, we can provide an answer for a situation where the diffusion operator has the following structure: $\varepsilon^\beta \,\partial^2_{x_1 x_1} + \varepsilon \, \Delta_{\overline{x}_1}$ or
\begin{equation*}
  \mathbb{D}_\varepsilon(x)  =
\left(
\begin{matrix}
 \varepsilon^\beta \, a_{11}(x_1) & 0 & 0 \\[2mm]
  0 &\varepsilon\, a_{22}(\frac{\overline{x}_1}{\varepsilon}) &  \varepsilon \, a_{23}(\frac{\overline{x}_1}{\varepsilon}) \\[2mm]
  0 &  \varepsilon\, a_{32}(\frac{\overline{x}_1}{\varepsilon}) &  \varepsilon\, a_{33}(\frac{\overline{x}_1}{\varepsilon})
\end{matrix}
\right),
\end{equation*}
and the convective field satisfying
\begin{equation}\label{str_1+}
\overrightarrow{V_\varepsilon}=
\left( \varepsilon^{\frac{\beta -1}{2}} v_1(s, x_1,t), \ \varepsilon  v_2(x_1, \tfrac{\overline{x}_1}{\varepsilon},t), \
 \varepsilon  v_3(x_1,\tfrac{\overline{x}_1}{\varepsilon},t)
\right),
\end{equation}
where $x \in \ \Omega_\varepsilon, \ s \in \Bbb R, \ t\in [0,T].$
Then, the  P\'eclet number is of  order $\mathcal{O}(\varepsilon^{-\frac{\beta +1}{2}})$ in  the longitudinal direction of the cylinder.

 In this case, the regular ansatz of the asymptotics for the solution $u_\varepsilon$ is as follows
\begin{align}\label{regul+}
  \mathcal{U}_\varepsilon(x,t) := & \ w_0(x_1,t) + \varepsilon^{\frac{\beta -1}{2}}  w_{\frac{\beta -1}{2}}(x_1,t)
  + \varepsilon\Big(w_1(x_1,t)  + u_1\big(x_1, \tfrac{\overline{x}_1}{\varepsilon}, t \big)  \Big) \notag
  \\
  & +   \varepsilon^{\frac{\beta +1}{2}} \Big(w_{\frac{\beta +1}{2}}(x_1,t)  + u_{\frac{\beta +1}{2}}\big(x_1, \tfrac{\overline{x}_1}{\varepsilon}, t \big)  \Big) +  \varepsilon^2 u_2\big(x_1, \tfrac{\overline{x}_1}{\varepsilon}, t \big)  + \varepsilon^{\frac{\beta +3}{2}} u_{\frac{\beta +3}{2}}\big(x_1, \tfrac{\overline{x}_1}{\varepsilon}, t \big).
\end{align}

Applying the same procedures as outlined in Sect.~\ref{Sec:4 Formal analysis}, it is evident that the coefficient $u_1$ solves
the Neumann problem \eqref{u_1}, but now $w_0$ is a solution to the Cauchy problem
\begin{equation}\label{limit_prob+Cauchy}
 \partial_t{w}_0   =  - \widehat{\varphi}({w}_0, x_1, t)\quad \text{in} \ \ (0, T],
    \qquad     w_0|_{t=0} =  0,
\end{equation}
where the variable $x_1$ is regarded as a parameter. According to condition ${\Cone},$ this problem has a unique solution and
$w_0 =0$ for $x_1 \in [0, \delta_1].$ By condition ${\Cthree},$ namely $\varphi|_{t=0} = 0,$ we have $\partial_t{w}_0|_{t=0} =0.$
The second  term in  \eqref{regul+}  depends on the parameter $\beta$:
this is $\varepsilon^{\frac{\beta -1}{2}}  w_{\frac{\beta -1}{2}}$ if $\beta \in (1, 3),$ and this is $\varepsilon w_1$  if $\beta \ge 3.$

Let us restrict ourselves to the  case $\beta =3$ in this paragraph to avoid  writing new problems for the coefficients.
In this case, ansatz \eqref{regul+} coincides with \eqref{regul}. The coefficient $u_2$  solves the Neumann problem \eqref{u_2} but  the summand $\partial_{x_1}\!\big(\big(v_1({w}_0,x_1,t) +   \partial_s v_1({w}_0,x_1,t)\, {w}_0 \big) {u}_1\big)$ in the differential equation will be absent. The coefficient $w_1$ is a unique solution to the Cauchy problem
\begin{equation*}
 \partial_t{w}_1   =  - \partial_s \widehat{\varphi}(w_0, x_1, t)  \, w_{1}-  \partial_{x_1}\!\big(v_1(w_0, x_1,t)\, {w}_0 \big) + f_1(x_1,t) \quad \text{in} \ \ (0, T],
 \qquad   {w_1}|_{t=0} =  0,
 \end{equation*}
where $f_1$ is defined by  \eqref{fun_1} but without the term $\partial^2_{x_1 x_1}{w}_0.$ Again due to  ${\Cone},$
$w_1 =0$ for $x_1 \in [0, \delta_1].$
In order to show that  $u_2|_{t=0} = 0$, we need to establish only that $\partial_t  u_1|_{t=0} = 0.$  It follows from \eqref{u_1} that this equality is  fulfilled if \eqref{new-cond} holds. Hence, $\mathcal{U}_\varepsilon|_{t=0} = 0.$

The boundary-layer ansatz for $\beta >1$ is constructed by the same asymptotic scale as in the ansatz \eqref{regul+}, but the main novelty is a new variable
$$
\zeta_1 = \frac{\ell -x_1}{\varepsilon^{\frac{1+\beta}{2}}}.
$$
In particular,  in the case $\beta =3$ the boundary-layer ansatz is as follows
\begin{equation*}
\mathcal{B}_\varepsilon(x,t)
 := \Pi_0\Big(\frac{\ell - x_1}{\varepsilon^2}, \frac{\bar{x}_1}{\varepsilon}, t\Big) + \varepsilon \, \Pi_1\Big(\frac{\ell - x_1}{\varepsilon^2}, \frac{\bar{x}_1}{\varepsilon}, t\Big)
    +
    \varepsilon^{2} \, \Pi_{2}\Big(\frac{\ell - x_1}{\varepsilon^2}, \frac{\bar{x}_1}{\varepsilon}, t\Big),
 \end{equation*}
where $\Pi_0,$ $\Pi_1$ and $\Pi_2$ are solutions to problems \eqref{prim+probl+0} and \eqref{prim+probl+1}, respectively.

Next we construct the  approximation $\mathfrak{A}_\varepsilon =  \mathcal{U}_\varepsilon + \chi_\ell(x_1) \, \mathcal{B}_\varepsilon$ in $\Omega_\varepsilon^{T},$
and  show with additional assumptions  \eqref{new-cond} and \eqref{therd_cond}  that $\mathfrak{A}_\varepsilon$ leaves residuals of order $\mathcal{O}(\varepsilon^\beta)$ if $\beta \in (1,2)$ and of order $\mathcal{O}(\varepsilon^2)$ if $\beta \ge 2.$ Applying the same reasoning as in Theorem~\ref{Th_1}, we can derive the corresponding asymptotic estimates.  Here we formulate the results for $\beta \ge 3.$ In this case, the two-term asymptotic approximation reads
$$ 
\mathfrak{A}_{1,\varepsilon}(x,t)  :=  w_0(x_1,t) +  \varepsilon\Big(w_1(x_1,t)  + u_1\big(x_1, \tfrac{\overline{x}_1}{\varepsilon}, t \big)  \Big)
 + \chi_\ell(x_1) \bigg(\Pi_0\Big(\frac{\ell - x_1}{\varepsilon^{\frac{1+\beta}{2}}}, \frac{\bar{x}_1}{\varepsilon}, t\Big) + \varepsilon \, \Pi_1\Big(\frac{\ell - x_1}{\varepsilon^{\frac{1+\beta}{2}}}, \frac{\bar{x}_1}{\varepsilon}, t\Big) \bigg)
$$ 
for $(x,t) \in \overline{\Omega}_\varepsilon \times [0,T],$
where  the coefficients are defined above in this section and the function $\chi_\ell$ in \eqref{cut-off_functions}.
\begin{theorem}
  Assume that, in addition to the main assumptions made in Section~\ref{Sec:Statement},
assumptions \eqref{new-cond} and \eqref{therd_cond} are satisfied. Then the asymptotic estimates \eqref{max_1+} and \eqref{app_estimate} hold for $T_1 = T.$ In addition, estimates \eqref{max_1+1} and \eqref{app_estimate+2} are also satisfied with $T_1 = T;$ in these estimates\\
\centerline{
$\mathfrak{A}_{0,\varepsilon} =  w_0(x_1,t) + \chi_\ell(x_1) \, \big(q_{\ell}(t) - w_{0}(\ell,t)\big) \,
  \exp\big(- \mathrm{v}_1(t) \, (\ell -x_1)/\varepsilon^{\frac{1+\beta}{2}}\big).$}
\end{theorem}

It can be concluded from the above that this asymptotic approach is applicable to convection-dominated  transport problems
with nonlinear diffusion operators in the form of
$\varepsilon^\beta \,\partial_{x_1}\big(a_1(x,t,u_\varepsilon, \partial_{x_1}u_\varepsilon)\big)  + \varepsilon \, \Delta_{\overline{x}_1} u_\varepsilon,$ where the positive and smooth function $a_1$ is equal to $\partial_{x_1}u_\varepsilon$ for $x_1 \in [0, \delta_1]$
(of course, this function requires some additional assumptions (see \cite[Chapt. V]{Lad_Sol_Ura_1968})).

\smallskip

{\bf 3.} Problem \eqref{probl} can be considered in a thin graph-like network. The analysis requires special assumptions about the convective vector field $\overrightarrow{V_\varepsilon}$ in neighborhoods of the network nodes and the construction of the node-layer terms of the asymptotics. More detailed information can be found in our article \cite{Mel-Roh_JMAA-2024}. In the same  paper and in \cite{Mel-Roh_AsAn-2024},  the Robin-type boundary condition \eqref{intr.2} has been considered with a multiplicative  intensity factor $\varepsilon^\alpha \, (\alpha \in \Bbb R)$ in front of     $\varphi_\varepsilon.$
As follows from \cite{Mel-Roh_AsAn-2024,Mel-Roh_JMAA-2024}, the most interesting case occurs at $\alpha=1$ because then the chemistry is in equilibrium with the flow in the limit as $\varepsilon \to 0.$ Therefore, we considered this case for our present problem \eqref{probl}.  As a result, the  homogenized transformation $-\widehat{\varphi}$ appears on  the right side of the limit problem \eqref{limit_prob}.

 An  intensity factor $\varepsilon^\tau,$ where $\tau > 1,$ can also be presented in the transversal components of the convective vector field \eqref{str_1}. Then the summands with the function $\overline{V}$ will be absent in problem \eqref{u_1} and appear in a next problem, e.g. if $\tau =2,$ these summands appear in problem \eqref{u_2}.

\smallskip

{\bf 4.} The work's results show that for mathematical models it is very important to make nondimensionalization and identify small parameters in the resulting problem. Then, using asymptotic methods, it is possible to construct approximations to solutions, even for nonlinear problems, saving time and resources for numerical calculations.

Now we can assert  that for convection-diffusion problems in thin cylinders,  the limit concentration
\begin{itemize}
  \item is governed by Taylor dispersion if  the coefficients both near $\partial^2_{x_1 x_1}u_\varepsilon$ and  the longitudinal component of the convective flow are of order $\mathcal{O}(1);$
  \item is a solution to the mixed hyperbolic quasilinear  problem \eqref{limit_prob}  for the original  problem \eqref{probl};
  \item is a solution to Cauchy problem \eqref{limit_prob+Cauchy} for the diffusion operator $\varepsilon^\beta \,\partial^2_{x_1 x_1} + \varepsilon \, \Delta_{\overline{x}_1}$ and the convective field \eqref{str_1+}.
\end{itemize}


\subsection*{Acknowledgment}
The authors  thank  for  funding by the  Deutsche Forschungsgemeinschaft (DFG, German Research Foundation) – Project Number 327154368 -- SFB 1313.

\end{document}